\documentclass[hidelinks,onefignum,onetabnum]{siamart250211}
\usepackage{amsfonts}
\usepackage{graphicx}
\usepackage{xspace}
\usepackage{xfrac}
\usepackage{relsize}
\usepackage{mathtools}
\usepackage{bm}
\usepackage{chemformula}
\usepackage{booktabs}
\usepackage[noend]{algpseudocode}
\usepackage[shortlabels]{enumitem}
\usepackage{xcolor}
\usepackage{multirow}
\usepackage{svg}
\usepackage{subcaption}
\usepackage{rotating}
\usepackage{overpic}
\usepackage{tikz}
\usepackage{xr-hyper}
\usetikzlibrary{arrows.meta}

\newsiamremark{remark}{Remark}
\newsiamremark{hypothesis}{Hypothesis}
\crefname{hypothesis}{Hypothesis}{Hypotheses}
\newsiamthm{claim}{Claim}

\usepackage{todonotes}

\title{
    Efficient Krylov methods for linear response
    in plane-wave electronic structure calculations%
\thanks{Submitted to the editors on DATE
\funding{MFH acknowledges support by the Swiss National Science Foundation
    (SNSF, Grant No.~221186) as well as the NCCR MARVEL, a National Centre of
    Competence in Research, funded by the SNSF (Grant No.~205602).
}}}
\author{
Michael F.~Herbst%
\thanks{Mathematics for Materials Modelling, Institute of
Mathematics \& Institute of Materials, École Polytechnique Fédérale de
Lausanne, 1015 Lausanne, Switzerland (\email{michael.herbst@epfl.ch})}
\and
Bonan Sun\footnotemark[2] \thanks{
    Max Planck Institute for Dynamics of Complex Technical Systems, 39106 Magdeburg, Germany
    (\email{bsun@mpi-magdeburg.mpg.de})}
}

\colorlet{dkgreen}{green!60!black}

\newcommand{\R}{\mathbb{R}}
\newcommand{\C}{\mathbb{C}}
\newcommand{\N}{\mathbb{N}}
\newcommand{\Z}{\mathbb{Z}}
\newcommand{\D}{\mathrm{d}}
\newcommand{\e}{\mathrm{e}}
\newcommand{\im}{\mathrm{i}}
\newcommand{\Real}{\operatorname{Re}}
\newcommand{\spann}{\operatorname{span}}
\newcommand{\range}{\operatorname{range}}
\newcommand{\dd}{\operatorname{d}}

\newcommand{\diag}{\operatorname{diag}}
\renewcommand{\vec}{\bm}

\newcommand{\Nocc}{N_{\mathrm{occ}}}
\newcommand{\Nocck}{N_{\mathrm{occ},\bm{k}}}

\newcommand{\Vext}{V_{\mathrm{ext}}} 
\newcommand{\Vxc}{V^{\mathrm{xc}}} 
\newcommand{\VH}{V^{\mathrm{H}}} 
\newcommand{\VHxc}{V^{\text{Hxc}}} 

\newcommand{\fFD}{f_{\mathrm{FD}}}
\newcommand{\epsilonF}{\epsilon_{\mathrm{F}}}
\newcommand{\FKS}{\mathcal{F}_{\mathrm{KS}}}
\newcommand{\KHxc}{K^{\mathrm{Hxc}}}
\newcommand{\La}{\mathbb{L}}
\newcommand{\Ecut}{E_{\mathrm{cut}}}
\newcommand{\G}{\mathbb{G}}
\newcommand{\X}{\mathbb{X}}
\newcommand{\Nb}{N_{\mathrm{b}}}
\newcommand{\Ng}{N_{\mathrm{g}}}
\newcommand{\Cd}{C_{\mathrm{d}}}
\newcommand{\taucg}{\tau^{\mathrm{CG}}}

\newcommand{\Nham}{N^{\mathrm{Ham}}}

\newcommand{\releta}{\eta_{\mathrm{rel}}}
\newcommand{\mpgrid}{\mathcal{G}_\text{MP}}
\newcommand{\Nmp}{N_{\text{MP}}}

\newcommand{\rend}{\bm r_\text{end}}
\newcommand{\xend}{\bm x_\text{end}}

\newcommand{\braket}[2]{\left\langle #1, #2 \right\rangle}

\newcommand{\abs}[1]{\left\vert #1 \right\vert}

\algnewcommand{\IIf}[1]{\State\algorithmicif\ #1\ \algorithmicthen}
\algnewcommand{\EndIIf}{\unskip\ }

\newcommand{\sbal}{\texttt{bal}\xspace}
\newcommand{\sagr}{\texttt{agr}\xspace}
\newcommand{\sgrt}{\texttt{grt}\xspace}
\newcommand{\sdten}{\texttt{D10}\xspace}
\newcommand{\sdhun}{\texttt{D100}\xspace}
\newcommand{\sdnorm}{\texttt{D10n}\xspace}

\newcommand{\sPbal}{\texttt{Pbal}\xspace}
\newcommand{\sPagr}{\texttt{Pagr}\xspace}
\newcommand{\sPgrt}{\texttt{Pgrt}\xspace}
\newcommand{\sPdten}{\texttt{PD10}\xspace}
\newcommand{\sPdhun}{\texttt{PD100}\xspace}
\newcommand{\sPdnorm}{\texttt{PD10n}\xspace}

\newcommand{\Hrho}{\mathcal{H}_\rho}
\newcommand{\HrhoB}{\bm{\mathcal{H}}_{\bm \rho}}

\ifpdf
\hypersetup{
  pdftitle={Inexact Krylov response},
  pdfauthor={M. F. Herbst, B. Sun}
}
\fi

\newif\ifarXiv\arXivfalse
 \arXivtrue

\ifarXiv
\newcommand{\supplement}{Appendix\xspace}
\else
\newcommand{\supplement}{Supplementary Material\xspace}
\fi

\begin{document}

\maketitle
\begin{abstract}
We propose a novel algorithm based on inexact GMRES methods
for linear response calculations in density functional theory.
Such calculations require iteratively
solving a nested linear problem $\mathcal{E} \delta\rho = b$
to obtain the variation of the electron density $\delta \rho$.
Notably each application of the dielectric operator $\mathcal{E}$
in turn requires the iterative solution of multiple linear systems,
the Sternheimer equations.
We develop computable bounds to estimate the accuracy of the density variation
given the tolerances to which the Sternheimer equations have been solved.
Based on this result we suggest reliable strategies
for adaptively selecting the convergence tolerances of the Sternheimer equations,
such that each application of $\mathcal{E}$ is no more accurate than needed.
Experiments on challenging materials systems of practical relevance
demonstrate our strategies
to achieve superlinear convergence as well as a reduction
of computational time by about 40\%
while preserving the accuracy of the returned response solution.
Our algorithm seamlessly combines with standard preconditioning approaches
known from the context of self-consistent field problems
making it a promising framework for efficient response solvers
based on Krylov subspace techniques.
\end{abstract}

\begin{keywords}
    Electronic-structure theory,
    Linear response theory,
    Inexact Krylov subspace methods,
    Inexact GMRES,
    Nested linear problems,
    Density-functional theory,
    Density-functional perturbation theory
\end{keywords}
\begin{MSCcodes}
    65F10, %
    65Z05, %
    81V55, %
    81-08, %
    81Q15 %
\end{MSCcodes}

\section{Introduction}
First-principle simulations of the electronic structure are a crucial
ingredient to drive research in physics, chemistry, materials science and engineering.
Due to their favourable balance of cost and accuracy
Kohn-Sham density-functional theory (DFT) calculations
are particularly popular
and nowadays routinely run in the millions in scientific workflows%
~\cite{%
kirklin2015open,Chanussot2021}
to predict material properties. %
Notably experiments probe material properties by observing
how the material changes under a variation of external conditions,
such as an applied field or an imposed strain.
Computationally this can be modelled as the response of a material's ground state (GS)
to a perturbation of the external potential.
In this sense materials properties can be identified
with derivatives of the ground state
with respect to external parameters.
For example interatomic forces are obtained as the derivative of the GS energy versus nuclear positions
and polarisabilities as the second derivative of the energy with respect to the electric field.
Other properties (such as phonon or Raman spectra) require
even high-order energy derivatives~\cite{baroni2001phonons,rivano2024density,windl1993second}. %
Beyond physical observables more recent applications of DFT derivatives
include (a) the design of machine-learned~(ML) exchange-correlation functionals%
~\cite{%
Kirkpatrick2021,Kasim2021,Chen2021,Wu2023xcml}
--- where derivatives of the GS wrt.~the ML parameters are needed ---
and (b) the computation of \textit{a posteriori} error estimates~\cite{herbst2020posteriori,dusson2017posteriori}
--- where derivatives wrt.~the perturbation induced by considering a larger discretisation basis
are required.

The application of perturbation theory to the special case of DFT
--- with the goal to compute the aforementioned GS derivatives ---
is known as density functional perturbation theory (DFPT)~\cite{baroniGreenSfunctionApproach1987,
gonzeDensityfunctionalApproachNonlinearresponse1989,gonze1995adiabatic,Gonze1995}.
See~\cite{normanPrinciplesPracticesMolecular2018} for its applications in quantum chemistry,
\cite{baroni2001phonons} for the application to phonons in crystal systems,
and~\cite{Lin2017,cances2023numerical} for recent performance improvements
motivated by analytical study.
See also~\cite{cancesMathematicalPerspectiveDensity2014} for a mathematical analysis in the context
of the reduced Hartree-Fock model, a model closely related to DFT.

In particular the computation of the \textit{linear response} $\delta\rho$
to the electronic density under a perturbation $\delta V_0$ of the external
potential plays a pivotal role in DFPT:
due to Wigner's $(2n+1)$ rule computing the $2n$-th and $(2n+1)$-st
derivatives of the GS energy requires computing the first $n$ derivatives of the GS density~\cite{gonze1995adiabatic,Gonze1995}.
Therefore efficiently solving the linear response problem is required
for computing the \textit{second and all higher order} GS energy derivatives.

Computing the linear response mapping $\delta V_0 \mapsto \delta\rho$
requires solving the Dyson equation
\begin{equation}
    \mathcal{E} \delta \rho = \chi_0 \delta V_0.
    \label{eqn:dyson}
\end{equation}
It involves the independent particle susceptibility $\chi_0$ as well as
$\mathcal{E}$, the adjoint of the dielectric operator,
which itself also depends on $\chi_0$.
Crucially, a single application of $\chi_0$ (thus also $\mathcal{E}$)
requires itself solving $\Nocc$ linear systems,
the Sternheimer equations~(SE).
Here $\Nocc$,
the number of DFT eigenstates with non-zero occupation,
notably scales linearly with the number of electrons in the material.

In this work we consider in particular the setting of ``large basis sets'',
such as plane-wave discretisations, where both the Dyson equation as well as the SE
need to be solved with iterative methods.
A delicate issue for obtaining an efficient iterative solver for this
nested set of problems is the selection of the convergence criteria
when solving the inner SEs:
due to the sheer number of solves too tight a tolerance quickly drives up computational cost,
while on the other hand too loose a tolerance can lead to inaccurate responses $\delta\rho$.
Moreover for some materials (such as metals) both the Dyson equation
as well as the inner SEs can become poorly conditioned~\cite{cances2023numerical},
such that (a) tight convergence tolerances may be difficult to achieve
and (b) numerical noise may amplify.
These aspects illustrate why efficient response calculations
remain a challenge in practical electronic structure simulations~\cite{Gonze2024}
--- in particular for metals and for disordered materials (where $\Nocc$ is large).

Our main contribution in this work is an algorithm for solving
the Dyson equation~\eqref{eqn:dyson},
in which for each individual SE the convergence threshold
is adaptively selected to reduce overall computational cost.
This algorithm is backed up by a theoretical analysis~(Theorem~\ref{thm:inexact_gmres_dyson})
demonstrating that despite using inexact solves the algorithm still obtains
a solution to the Dyson equation within the prescribed accuracy.
A key ingredient is a novel rigorous bound for the error in the application
of the linear operator $\mathcal{E}$ resulting from inexact SE solves~(Lemma \ref{lem:cgtol_E}),
combined with the inexact GMRES framework of Simoncini and Szyld~\cite{simoncini2003theory}.
For the latter our work also develops an extension
to estimate the smallest singular value of the upper Hessenberg matrix
on the fly, which is generally applicable beyond the Dyson equation context.

Beyond our guaranteed bound we further develop a series of physically motivated approximations,
which enable additional savings in computational cost while
hardly impacting accuracy~(Section~\ref{sec:practical}).
We further analyse scaling of these strategies with increasing system size
to ensure that our results do not deteriorate %
when targeting larger materials.
In experiments on non-trivial material systems ---
a metallic supercell and a Heusler alloy material ---
our recommended \sbal (balanced) algorithm achieves a reduction in the computational
time of about 40\% compared to approaches with fixed SE tolerances of similar accuracy.

Our developments complement recent techniques to improve DFPT convergence.
Throughout this work we will employ the Schur complement approach suggested
in~\cite{cances2023numerical} for the SE. On metallic systems we additionally
employed the Kerker preconditioner within the inexact GMRES
for solving the Dyson equations --- leading to a speed-up of about a factor of $3$
if all techniques are combined.
Surprisingly, despite Kerker being widely established for accelerating
convergence towards the DFT ground state solution,
its use as a preconditioner
when solving the Dyson equation does not seem
to be widely established in the physics literature,
one exception being the discussion in~\cite{Gerhorst2024}.

While developed for the setting of plane-wave DFT we remark that our ideas can
be easily extended to related mean-field models such as
Hartree-Fock~\cite{lin2019math_electronic}
or the Gross-Pitaevskii equations (GPE)~\cite{bao2013mathematical}
due to the closely related mathematical structure~\cite{cances2010naofnep}.
Moreover our modifications to the inexact GMRES framework are general and 
could be applied to other nested linear problems.

The remainder of this paper is structured as follows.
In Section~\ref{sec:mathematical}, we review the mathematical formulation of DFT and DFPT
as well as its numerical discretisation in plane-wave basis.
In Section~\ref{sec:theory}, we derive computable error bounds quantifying how
inaccuracies in Sternheimer solutions propagate to the solution of the Dyson
equation as well as practical strategies for selecting SE convergence thresholds.
Numerical experiments confirming the efficiency gains of our algorithm
are presented in Section~\ref{sec:experiments}.

\section{Mathematical framework}
\label{sec:mathematical}

\subsection{Kohn-Sham density functional theory for periodic systems}
\label{sec:DFT}
We consider a simulation cell $\Omega = [0,1) \bm a_1 + [0,1) \bm a_2 + [0,1) \bm a_3$
for some not necessarily orthonormal basis $(\bm a_1, \bm a_2, \bm a_3)  \in \R^3$
and denote the associated periodic lattice in position space by
$\La = \Z \bm a_1 + \Z \bm a_2 + \Z \bm a_3$ as well as reciprocal lattice by
\begin{equation}
  \La^* = \Z \bm b_1 + \Z \bm b_2 + \Z \bm b_3, \quad \text{with} \quad \bm b_i \cdot \bm a_j = 2\pi \delta_{ij}, \quad i,j = 1,2,3.
\end{equation}
The space of locally square-integrable $\La$-periodic functions is given by
\begin{equation}
  L^2_{\Omega}(\R^3, \C) = \left\{ u: \R^3 \to \C \mid u(\bm r + \bm R) = u(\bm r), \ \forall \bm r \in \R^3, \bm R \in \La, \ \int_{\Omega} |u(\bm r)|^2 \dd \bm r < \infty \right\}
\end{equation}
endowed with the usual inner product $\langle\,\cdot\,,\,\cdot\,\rangle$.
We consider the case of a system containing an even $N \in \N$ number of electrons in a spin-paired state
at finite smearing temperature $T > 0$.
In atomic units the Kohn-Sham Hamiltonian reads as
\begin{equation}
    \Hrho = - \frac{1}{2}\Delta + \Vext(\vec{r}) + \VHxc_\rho(\vec{r}),
    \label{eqn:Hamiltonian}
\end{equation}
where $\Vext$ is the external potential modelling the interaction of the electrons with the environment
(e.g.~the potential generated by the nuclei of the material), and $\VHxc_\rho(\vec{r}) = \VH_\rho(\vec{r}) + \Vxc_\rho(\vec{r})$
is the Hartree-exchange-correlation potential modelling the electron-electron interactions.
Both $\VHxc_\rho$ and $\Vext$ are $\La$-periodic functions.
The Hartree potential $\VH_\rho$ is obtained as the unique zero-mean solution of the
periodic Poisson equation $-\Delta \VH_\rho(\vec{r}) = 4\pi \left( \rho(\vec{r}) - \frac1{|\Omega|} \int_\Omega \rho \D r \right)$
and $\Vxc_\rho$ is the exchange-correlation potential, both depending on $\rho$.
We remark that $\Hrho$ is a self-adjoint operator on $L^2_{\Omega}$ with domain $H^2_{\Omega}$, bounded below and with compact resolvent such that its spectrum is composed of a nondecreasing sequence of real eigenvalues
$(\epsilon_n)_{n\in\N}$ accumulating at $+\infty$. 
The DFT ground-state is defined by an $L^2_{\Omega}$-orthonormal
set of orbitals $(\phi_n)_{n\in\N}$ obtained as eigenfunctions of the Kohn-Sham Hamiltonian $\Hrho$
\begin{equation}
    \Hrho \phi_n(\bm r) = \epsilon_n \phi_n(\bm r), \quad
    \epsilon_1 \leq \epsilon_2 \leq \cdots, \quad
    \langle \phi_n, \phi_m \rangle = \delta_{nm}.
    \label{eq:KS}
\end{equation}

The ground-state density $\rho$ is defined in terms of the Kohn-Sham orbitals $(\phi_n)_{n\in\N}$ %
\begin{equation}
    \rho(\vec r) = \sum_{n=1}^{\infty} f_n \left| \phi_n(\bm r) \right|^2
    \qquad \text{with} \qquad
    f_n = \fFD\left(\dfrac{\epsilon_n - \epsilonF}{T}\right),
    \label{eqn:density}
\end{equation}
where $f_n \in [0, 2)$ are the occupation numbers and
$\fFD$ is the occupation function.
The Fermi level $\epsilonF$ is chosen as the real number such that
    $\int_\Omega \rho(\vec{r}) \,\D r = \sum_{n=1}^{\infty} f_n = N$.
For simplicity we consider the case of a monotonic smearing function,
such as the Fermi-Dirac function $\fFD(x) = 2/(1+e^x)$.

Since $\Hrho$ depends on $\rho$, which in turn depends on the eigenfunctions $\phi_n$ of $\Hrho$,
\eqref{eq:KS} is a non-linear eigenvalue problem, which is usually equivalently written as a
fixed-point problem in the density
\begin{equation}
    \label{eqn:fixedpoint}
    \rho = \FKS\left(\Vext + \VHxc_\rho \right).
\end{equation}
Here we introduced the potential-to-density map
\begin{equation}
    \label{eqn:potdensmap}
    \FKS: \quad V \mapsto \rho = \sum_{n=1}^{\infty} \fFD\left(\dfrac{\epsilon_n - \epsilonF}{T}\right) \left| \phi_n \right|^2,
\end{equation}
with $(\phi_n)_{n\in\N}$ being the $L^2_{\Omega}$-orthonormal eigenfunctions of $- \frac{1}{2}\Delta + V$
ordered by non-decreasing corresponding eigenvalues $(\epsilon_n)_{n\in\N}$.
The form \eqref{eqn:fixedpoint} of the non-linear eigenvalue problem \eqref{eq:KS}
is called self-consistent field (SCF) problem and is usually solved using
(preconditioned, accelerated) fixed-point iterations,
see~\cite{lin2019KSDFT,cances2021convergence,Herbst2022,cances2023KSmodels} for details.
Notably these methods require in each step the evaluation of $\FKS$ on a given iterate $\rho_n$,
which in turn implies a diagonalisation of $H_{\rho_n}$.
For a mathematical presentation of SCF methods see~\cite{cances2021convergence,lin2019math_electronic,cances2023KSmodels}
and consider~\cite{herbstBlackboxInhomogeneousPreconditioning2020,Herbst2022} for an in-depth discussion
of the practical challenges of obtaining robust SCF algorithms.

\subsection{Density functional perturbation theory (DFPT)}
\label{sec:dfpt}

Having solved the SCF problem \eqref{eqn:fixedpoint}
we obtain the DFT ground-state density $\rho$.
Notably \eqref{eqn:fixedpoint} and thus the solution $\rho$ depends on $\Vext$,
defining thus an implicit function $\Vext \mapsto \rho$
mapping from the external potential $\Vext$ to the self-consistent density $\rho$.
The goal of DFPT is to compute derivatives of this function,
which thus allows to study how the ground state density responds to perturbations
in the external conditions of the electrons
(displacement of nuclei, changes to the electromagnetic field, etc.).

Let $\delta V_0$ denote an infinitesimal local perturbation of $\Vext$,
i.e.~a perturbation that can be written as a multiplication operator
involving the periodic function $\vec{r} \mapsto \delta V_0(\vec{r})$.
We now compute the linearisation of \eqref{eqn:fixedpoint} using the chain rule
leading to the Dyson equation
\begin{equation}\label{eq:dyson_equation}
    \delta \rho = \underbrace{\FKS'(\Vext + \VHxc_\rho)}_{\chi_0}
        \left( \delta V_0 + \KHxc_\rho \delta \rho \right)
    \iff \left( 1 - \chi_0 \KHxc_\rho \right) \delta \rho = \chi_0 \delta V_0,
\end{equation}
where the independent-particle susceptibility $\chi_0$ is the derivative of $\FKS$
and Hartree-exchange-correlation kernel $\KHxc_\rho$ is the derivative
of $\rho \mapsto \VHxc_\rho$,
both evaluated at the self-consistent potential and density, respectively.
Given an infinitesimal perturbation $\delta V_0(\bm{r})$
the first-order approximation of the variation of the electron density $\delta \rho$
is thus obtained by solving the Dyson equation linear system \eqref{eq:dyson_equation}.
Notably the operator
\begin{equation}
    \mathcal{E} := 1 - \chi_0 \KHxc_\rho
\end{equation}
turns out to be the adjoint of the dielectric operator $1 - \KHxc_\rho \chi_0$%
~\cite{adler1962quantum,herbstBlackboxInhomogeneousPreconditioning2020}.
In the physics literature the dielectric operator is usually
given the symbol $\varepsilon$.
For notational simplicity we drop the explicit indication of the adjoint
in this work and employ the capitalised symbol
$\mathcal{E}$ to directly denote the dielectric \textit{adjoint}
as introduced above.

As we will detail in Section~\ref{sec:structure_dyson},
the Dyson equation \eqref{eq:dyson_equation} is typically solved iteratively,
such that an efficient application of both $\KHxc_\rho$ and $\chi_0$ is crucial.
The kernel $\KHxc_\rho = K^\text{H}_\rho + K^\text{xc}_\rho$
is the sum of two terms, the Coulomb kernel $K^\text{H}_\rho$
and the exchange-correlation kernel $K^\text{xc}_\rho$.
The former is diagonal in a plane-wave basis and the latter
can be applied using a real-space convolution, thus both are indeed efficient to treat.

In contrast, applying the independent particle
susceptibility $\chi_0 : \delta V \mapsto \delta\rho_0$ is considerably
more challenging
and --- if done exactly --- would in fact involve an infinite
sum over all eigenpairs $(\epsilon_n, \phi_n)$ of $\Hrho$.
However, typical smearing functions (such as $\fFD$) rapidly decay
to zero as $x\to\infty$.
This enables to truncate such infinite sums (which all involve factors $f_n$)
to only the finite set of occupied orbitals where $f_n > \epsilon_{\mathrm{occ}}$.
See~\cite{cances2023numerical} for an in depth discussion.
Denoting the number of occupied orbitals
by $\Nocc := \max\{n \in \N \mid f_n > \epsilon_{\mathrm{occ}}\}$
we can express the application of $\chi_0$ as
\begin{equation}\label{eq:chi0}
  \chi_0 \delta V
    = \sum_{n=1}^{\Nocc} 2 f_n \Real \big( \phi_n^* \, \delta \phi_n \big) + \delta f_n \, \left| \phi_n \right|^2
\end{equation}
where $\delta \phi_n$ is the change in the orbital $\phi_n$
induced by the perturbation $\delta V$ (detailed below)
and $\delta f_n$ is the change in occupation computed by
\begin{equation}\label{eq:delta_f}
    \delta f_n = f_n' \Big( \braket{\phi_n}{\delta V \phi_n} - \delta \epsilonF \Big).
\end{equation}
Here, $\delta \epsilonF$ is chosen such that $\sum_{n=1}^{\Nocc} \delta f_n = 0$
(charge conservation) and here and henceforth we use the convention
\begin{equation}\label{eq:fprime}
    f'_n \coloneqq \frac{f_n - f_n}{\epsilon_n - \epsilon_n} = \frac{1}{T}\fFD'\left(\frac{\epsilon_n - \epsilon_F}{T}\right).
\end{equation}

For obtaining $\delta \phi_n$ a few ways have been discussed in the literature
amounting to different gauge choices. Here we follow the
minimal gauge approach introduced in~\cite{cances2023numerical}
with the goal to obtain a numerically stable DFPT solution
also for metallic systems.
By a partition of unity with
\begin{equation}
    P = \sum_{n=1}^{\Nocc} | \phi_n \rangle \langle \phi_n |, \qquad\text{and}\qquad  Q = I - P = \sum_{n=\Nocc+1}^{\infty} | \phi_n \rangle \langle \phi_n |,
\end{equation}
--- respectively the orthogonal projectors onto the occupied subspace $\spann \{ \phi_n \}_{n=1}^{\Nocc}$
and unoccupied subspace $\spann \{ \phi_n \}_{n=\Nocc+1}^{\infty}$ ---
we decompose $\delta \phi_n$ into two contributions
$\delta \phi_n = P \delta \phi_n + Q \delta \phi_n$.
For each $n = 1,\ldots,\Nocc$ the first term is computed via an explicit sum over states
\begin{equation}\label{eq:delta_phi_P}
    \delta \phi_n^P := P \delta \phi_n = \sum_{m=1}^{\Nocc} \frac{f_n^2}{f_n^2 + f_m^2} \frac{f_n - f_m}{\epsilon_n - \epsilon_m} \braket{\phi_m}{\delta V \phi_n} \phi_m.
\end{equation}
while the second term is computed by solving the Sternheimer equations~(SE)
\begin{equation}
    Q \left(\Hrho - \epsilon_n \right) Q \delta \phi_n^Q = -Q \left(\delta V \phi_n\right)
    \label{eq:sternheimer_equation}
\end{equation}
for $\delta \phi_n^Q := Q \delta \phi_n$.
See~\cite{cances2023numerical}
for a detailed derivation of the above expressions.

\subsection{Numerical Discretisation}

\subsubsection{Plane-wave basis sets and Fourier transforms}\label{sec:plane_wave_basis}
In this section we provide a
description of plane-wave basis sets and Fourier transforms,
including key details that will be important for our development of
adaptive Sternheimer tolerance strategies.
A deeper theoretical or computational discussion of these standard results
can be found in~\cite{lin2019KSDFT,cances2022planewave}.

In periodic systems the unknown functions are $\La$-periodic and
thus can be approximated using a Fourier basis.
We use normalised plane waves, defined as
$e_{\bm G}(\bm r) = \e^{\im \bm G \cdot \bm r} / \sqrt{|\Omega|}$ where $|\Omega|$
is the volume of the unit cell $\Omega$ and $\bm G \in \La^*$.
To represent densities and orbitals to a consistent approximation
a larger basis set is employed for densities~\cite{cances2022planewave,lin2019KSDFT}.
We define the spherical Fourier basis $\mathcal{X}^{\bigcirc}$ (for orbitals)
and the cubic Fourier basis $\mathcal{X}^{\square}$ (for densities) by
\begin{equation}
    \label{eqn:fourier_bases}
  \begin{aligned}
    \mathcal{X}^{\bigcirc}&:= \spann \{ e_{\bm G} \mid \bm G \in \G^{\bigcirc} \} \ \text{with} \ \G^{\bigcirc} := \left\{ \bm G \in \La^* \mid \|\bm G\|_2 \leq \sqrt{2\Ecut} \right\},\\
    \mathcal{X}^{\square}&:= \spann \{ e_{\bm G} \mid \bm G \in \G^{\square} \} \ \text{with} \ \G^{\square} := \left\{ \bm G \in \La^* \mid \|\bm G\|_{\infty} \leq 2\sqrt{2\Ecut} \right\},
  \end{aligned}
\end{equation}
respectively. Here, the cut-off energy $\Ecut > 0$ is a user-defined parameter.
The dimension of these sets are
\begin{equation}\label{eq:Nb_Ng_definition}
  \begin{aligned}
    \Nb &:= \dim \mathcal{X}^{\bigcirc} = \left| \G^{\bigcirc} \right| = \Cd \dfrac{4 \pi}{3} \left( \sqrt{2\Ecut} \right)^3 \simeq \frac{\sqrt{2}}{3\pi^2} \Ecut^{3/2} \abs{\Omega}, \\
    \Ng &:= \dim \mathcal{X}^{\Box} = \left| \G^{\Box} \right| = \Cd \left( 4 \sqrt{2\Ecut} \right)^3 \simeq \frac{16\sqrt{2}}{\pi^3} \Ecut^{3/2} \abs{\Omega},
  \end{aligned}
\end{equation}
where $\Cd \simeq |\Omega|/(2\pi)^3$ is the (average) number of lattice points per unit volume in the reciprocal lattice $\La^*$. 

The discretised orbitals and electron density, denoted by the bold font $\bm \phi_n^{\bigcirc}$ and $\bm \rho^{\square}$, are vectors in $\C^{\Nb}$ and $\C^{\Ng}$ containing the Fourier coefficients
\begin{equation}
  \begin{gathered}
    \bm \phi_n^{\bigcirc} := \widehat{[\phi_n]}_{\bm G \in \G^{\bigcirc}} \in \C^{\Nb}, \ \phi_n(\bm r) \approx \sum_{\bm G \in \G^{\bigcirc}} \widehat{[\phi_{n}]}_{\bm G} e_{\bm G}(\bm r)\in \mathcal{X}^{\bigcirc}, \\
    \bm \rho^{\square} := \widehat{[\rho]}_{\bm G \in \G^{\square}} \in \C^{\Ng}, \ \rho(\bm r) \approx \sum_{\bm G \in \G^{\square}} \widehat{[\rho]}_{\bm G} e_{\bm G}(\bm r)\in \mathcal{X}^{\square},
  \end{gathered}
\end{equation}
respectively. Note that with our normalisation convention
the Fourier coefficients $\widehat{[u]}_{\bm G}$ of a function
$u(\bm r) \in L^2_{\Omega}(\R^3, \C)$ are obtained as
$\widehat{[u]}_{\bm G} = \langle e_{\bm G}, u \rangle_{L^2_{\Omega}(\R^3, \C)}$.
The reciprocal space grid $\G^{\Box}$ \eqref{eqn:fourier_bases}
can equivalently be expressed as
\begin{equation}
    \G^{\Box} = \left\{ i_1 \bm b_1 + i_2 \bm b_2 + i_3 \bm b_3 \mid -\frac{\Ng^x}{2} \leq i_1 < \frac{\Ng^x}{2}, -\frac{\Ng^y}{2} \leq i_2 < \frac{\Ng^y}{2}, -\frac{\Ng^z}{2} \leq i_3 < \frac{\Ng^z}{2} \right\},
\end{equation}
using even integers $\Ng^x, \Ng^y, \Ng^z$.
By means of the inverse Fourier transform~(FT) Fourier coefficients are mapped
to values on the respective real-space grid
\begin{equation}
    \X := \left\{ \dfrac{i_1}{\Ng^x} \bm a_1 + \dfrac{i_2}{\Ng^y} \bm a_2 + \dfrac{i_3}{\Ng^z} \bm a_3 \mid 0 \leq i_1 < \Ng^x, 0 \leq i_2 < \Ng^y, 0 \leq i_3 < \Ng^z \right\}.
\end{equation}
For example, denoting by $\bm W \in \C^{\Ng \times \Ng}$ the unitary discrete
Fourier transform matrix and introducing the weights
$w =  \sqrt{|\Omega|}/\sqrt{\Ng}$ the inverse FT is given by $w^{-1} \bf W^{-1}$.
It maps the Fourier coefficients $\bm \rho^{\Box}$
to function values $\rho(\bm r)$ on the grid $\X$.
The normalisation factor $w$ in this expression is chosen
such that the values obtained by an inverse FT are indeed
the unscaled function values.
Note, the value of the normalisation factor $w$ varies across different
software packages. The convention taken here is the one of DFTK%
~\cite{herbst2021dftk}.

To achieve the same transformation for the orbitals $\phi_n(\bm r)$
--- which are only represented on the (smaller) spherical grid $\G^{\bigcirc}$ ---
we first pad $\bm \phi_n^{\bigcirc}$ with zeros to get $\bm \phi_n^{\Box}$
and only then apply the inverse FT.
Specifically, let $\bm Z \in \C^{\Ng \times \Nb}$ be the zero-padding matrix,
the transformations from real space function values to the Fourier
coefficients on the spherical grid
(and vice versa) are given by: %
\begin{equation}\label{eq:fft_ifft}
  \begin{aligned}
    \bm F := w \bm Z^{\top} \bm W: \C^{\Ng} \to \C^{\Nb},               & \ [u(\bm r)]_{\bm r \in \X} \mapsto \widehat{[u]}_{\bm G \in \G^{\bigcirc}},  \\
    \bm F^{-1} := w^{-1} \bm W^{-1} \bm Z: \C^{\Nb} \to \C^{\Ng}, & \ \widehat{[u]}_{\bm G \in \G^{\bigcirc}} \mapsto [u(\bm r)]_{\bm r \in \X},
  \end{aligned}
\end{equation}
which can be implemented efficiently by leveraging the FFT algorithm.

\subsubsection{Notation for discretised quantities}
In this paper, we use the following conventions for denoting discretised quantities. 
In general, we use bold font of the same letter to denote the corresponding discretised object. 
Quantities related to orbitals ($\phi_n$, $\delta \phi_n$, $\Hrho$) are 
by default discretised in the spherical Fourier basis $\mathcal{X}_{\Ecut}^{\bigcirc}$ 
and quantities related to electron density and 
potentials ($\rho$, $\delta \rho$, $V$, $\delta V$, $\chi_0$) are by default 
discretised in real space $\X$. 
Therefore, we may omit the superscripts $\bigcirc$ and $\Box$ of the quantities 
in the following discussions for simplicity. 
See Figure~\ref{fig:discretisation} for a summary of notations. 
Note that if we let 
$\bm \Phi = \left[ \bm \phi_1, \cdots, \bm \phi_{\Nocc} \right] \in \C^{\Nb \times \Nocc}$ 
be the matrix of occupied orbitals, then the projectors $P$ and $Q$ can 
be written as $\bm P = \bm \Phi \bm \Phi^\mathsf{H}$ and $\bm Q = \bm I_{\Nb} - \bm P$
respectively. We also remark that the orbitals are \emph{orthonormal} 
in the Fourier basis, i.e., $\bm \Phi^\mathsf{H} \bm \Phi = \bm I_{\Nocc}$ but 
are only \emph{orthogonal} in the real space due to the normalisation factor $w$.

The corresponding operations $\Real$, $\odot$, $^*$, and $|\cdot|^2$ in
the following discrete settings should be understood component-wise. In particular,
$\odot$ denotes the element-wise product, known as the Hadamard product.

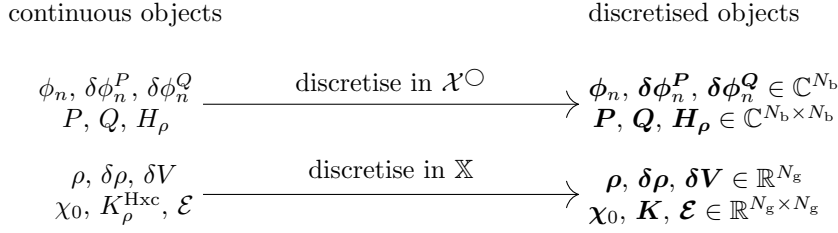
\begin{figure}
  \centering
  \begin{tikzpicture}
    \node[scale=1.0,anchor=east,align=center] (ti) at (0.5,1.2){
      continuous objects
    };
    \node[scale=1.0,anchor=west,align=center] (ti_dis) at (5,1.2){
      discretised objects
    };

    \node[scale=1.0,anchor=east,align=center] (ob) at (0,0){
      $\phi_n$, $\delta \phi_n^P$, $\delta \phi_n^Q$ 
      \\ $P$, $Q$, $\Hrho$
    };
    \node[scale=1.0,anchor=west,align=center] (ob_dis) at (5,0){
      $\bm \phi_n$, $\bm{\delta \phi^P}_n$, $\bm{\delta \phi^Q}_n \in \C^{\Nb} $
      \\ $\bm P$, $\bm Q$, $\HrhoB \in \C^{\Nb \times \Nb} $
    };
    \draw [-{>[scale=2]}] (ob) -- node[above=0.5mm]{discretise in $\mathcal{X}^{\bigcirc}$} (ob_dis);

    \node[scale=1.0,anchor=east,align=center] (ds) at (0,-1.2){
      $\rho$, $\delta \rho$, $\delta V$
      \\ $\chi_0$, $\KHxc_\rho$, $\mathcal{E}$
    };
    \node[scale=1.0,anchor=west,align=center] (ds_dis) at (5,-1.2){
      $\bm \rho$, $\bm{\delta \rho}$, $\bm{\delta V} \in \R^{\Ng}$
      \\ $\bm \chi_0$, $\bm K$, $\bm{\mathcal E} \in \R^{\Ng \times \Ng}$
    };
    \draw [-{>[scale=2]}] (ds) -- node[above=1mm] {{discretise in $\X$}} (ds_dis);

  \end{tikzpicture}
  \caption{Discretisation of continuous objects in the plane-wave basis sets.}
  \label{fig:discretisation}
\end{figure}

\subsubsection{Discretised Kohn-Sham equations}
Given the above setup, the discretised Kohn-Sham equations (c.f. \eqref{eq:KS}) can be written as
\begin{equation}\label{eq:discretised_KS}
  \begin{gathered}
    \HrhoB \bm \phi_n = \epsilon_n \bm \phi_n, \ n=1,\cdots, \Nocc, \ \langle \bm \phi_n, \bm \phi_m \rangle = \delta_{nm}, \ \bm \rho = \sum_{n=1}^{\Nocc} f_n | \bm F^{-1} \bm \phi_n |^2,
  \end{gathered}
\end{equation}
where for the computation of $\bm \rho$, we first transform the orbitals to
real space by the inverse FT and then sum up the modulus squares of the
transformed orbitals. This non-linear eigenvalue problem is usually solved
by the \emph{self-consistent field} (SCF) iterations yielding $\Nocc$ eigenpairs
\begin{equation}
    \left(\epsilon_n,  \bm \phi_n \right) \in \R \times \C^{\Nb} \quad \text{for} \quad n=1,\cdots, \Nocc,
\end{equation}
as well as the electron density $\bm \rho \in \R^{\Ng}$ and the Hamiltonian matrix $\HrhoB \in \C^{\Nb \times \Nb}$.

\subsubsection{Discretised Dyson equation}
With the Kohn-Sham equations \eqref{eq:discretised_KS} solved and all these quantities 
at hand we can write down the discretised Dyson equation as (c.f. \eqref{eq:dyson_equation})
\begin{equation}\label{eq:dyson_discretised}
  \text{Seek } \bm{\delta \rho} \in \R^{\Ng} \text{ such that } \bm{\mathcal E} \bm{\delta \rho} = \bm{\delta \rho}_0,
\end{equation}
where $\bm{\mathcal E} := \bm I_{\Ng} - \bm \chi_0 \bm K \in \R^{\Ng \times \Ng}$, $\bm{\delta \rho}_0 := \bm \chi_0 \bm{\delta V_0} \in \R^{\Ng}$. The dielectric adjoint $\bm{\mathcal E}$ is accessed through the matrix-vector product
(c.f. \eqref{eq:chi0})
\begin{equation}\label{eq:op_discretised}
    \bm{\mathcal E} \bm v =\bm v - \|\bm K \bm v\| \sum_{n=1}^{\Nocc} 2 f_n \Real \Big((\bm F^{-1} \bm{\phi}_n)^* \odot ( \bm F^{-1} \bm{\delta \phi^P}_n + \bm F^{-1} \bm{\delta \phi^Q}_n ) \Big) + \delta f_n \left| \bm F^{-1} \bm{\phi}_n \right|^2,
\end{equation}
where for $n=1,\cdots, \Nocc$, occupied variations $\bm{\delta \phi^P}_n \in \C^{\Nb}$ are computed explicitly by a sum over states (c.f. \eqref{eq:delta_phi_P})
\begin{equation}
    \label{eq:sumstateschi}
  \bm{\delta \phi^P}_n = \sum_{m=1}^{\Nocc} \frac{f_n^2}{f_n^2 + f_m^2} \frac{f_n - f_m}{\epsilon_n - \epsilon_m} \braket{\bm\phi_m}{\bm F \left( \left(\bm K \bm v\right) \odot \left( \bm F^{-1} \bm{\phi}_n \right) \right)} \bm\phi_m / \|\bm K \bm v\|,
\end{equation}
the unoccupied variations $\bm{\delta \phi^Q}_n \in \C^{\Nb}$ are computed by solving the discretised SE (c.f. \eqref{eq:sternheimer_equation})
\begin{equation}
    \label{eq:dicretized_sternheimer_equation}
  \bm Q \left(\HrhoB - \epsilon_n \bm I_{\Nb} \right) \bm Q \bm{\delta \phi^Q}_n = -\bm Q \bm F \left( \left(\bm K \bm v\right) \odot \left( \bm F^{-1} \bm{\phi}_n \right) \right)/\|\bm K \bm v\|,
\end{equation}
and the occupation changes $\delta f_n$ are computed by %
$\delta f_n = f_n' \left( \delta \epsilon_n - \delta \epsilonF \right)$
with $\delta \epsilon_n = \braket{\bm\phi_n}{\bm F \left( \left(\bm K \bm
v\right) \odot \left( \bm F^{-1} \bm{\phi}_n \right) \right)}/ \|\bm K \bm v\|$
and $\delta \epsilonF$, $f'_n$ as described in Section~\ref{sec:dfpt}.

Note, that this formulation exploits the linearity of $\chi_0$ to rewrite
$\bm{\mathcal E} \bm v
= \bm v - \|\bm K \bm v\| \bm \chi_0 \left( \bm K \bm v / \|\bm K \bm v\| \right)$
explaining the appearance of terms $\|\bm K \bm v\|$ in the above expressions.
This controls the size of the
RHS of \eqref{eq:dicretized_sternheimer_equation}
as well as the terms
in the sum over states formula \eqref{eq:sumstateschi},
avoiding in particular underflow for small $\bm K \bm v$.
We find this rescaling to improve overall
accuracy of the computed orbital variations
$\bm{\delta \phi}_n = \bm{\delta \phi^P}_n + \bm{\delta \phi^Q}_n$.

\subsubsection{Structure of the Dyson equation and typical numerical algorithms}\label{sec:structure_dyson}
Recall that $\Ng \approx 16 \Nb$ (see \eqref{eq:Nb_Ng_definition}). Therefore
in plane-wave bases the Dyson equation is a large $\Ng \times \Ng$
\emph{real non-symmetric} linear system,
while the SE is a \emph{complex Hermitian} linear system
with a smaller size $\Nb \times \Nb$.
Both the dielectric adjoint $\bm{\mathcal{E}}$ and the Hamiltonian
$\HrhoB$ are not explicitly formed
and are only accessed through matrix-vector products
to benefit from an efficient FFT-based evaluation.
Both the Dyson and Sternheimer equations are thus solved using iterative methods.

We use the generalised minimal residual (GMRES) method~\cite{saad1986gmres} for
the Dyson equation and the conjugate gradient (CG)~\cite{hestenes1952cg} method
for the SEs.
This results in a nested problem:
at the $i$-th GMRES iteration, an application of $\bm{\mathcal{E}}$ to a Krylov
basis vector is required, which in turn requires the iterative solution
of $\Nocc$ SEs by CG with some tolerance $\taucg_{i,n}$
for $n = 1,\cdots, \Nocc$.
Selecting the appropriate CG tolerances is crucial for the efficiency
and accuracy of the solution of the Dyson equation, which will be discussed
in the following sections.

\section{An inexact preconditioned GMRES method for the Dyson equation}
\label{sec:theory}
\subsection{Why a naive static tolerance selection can fail}
\label{sec:gmres_fail}
We first illustrate the deficiencies of a naive GMRES approach,
which solves the Dyson equation $\bm{\mathcal E} \bm{\delta \rho} = \bm{\delta \rho}_0$
employing simply a static threshold $\taucg_{i,n}$ for the inner SE problems.
Given an initial guess $\bm x_0 \in \R^{\Ng}$ to the solution $\bm{\delta \rho}$,
GMRES generates a sequence of approximations $\bm x_i = \bm x_0 + \bm V_i \bm y_i$ for $i = 1, \ldots, m$,
such that $\bm y_i$ minimises the $\ell_2$ norm of the
residual \mbox{$\bm r_i := \bm{\delta \rho}_0 - \bm{\mathcal E} \bm x_i$}.
Here, $m \ll \Ng$ is the maximal subspace size for GMRES,
$\bm V_i = \left[ \bm v_1, \bm v_2, \cdots, \bm v_i \right] \in \R^{\Ng \times i}$ is the
orthonormal basis of the Krylov subspace $\mathcal{K}_i(\bm{\mathcal E}, \bm
b)$ and satisfies the Arnoldi decomposition
$\bm{\mathcal E} \bm V_i = \bm V_{i+1} \bm H_i$
where $\bm H_i \in \R^{(i+1) \times i}$ is an upper Hessenberg matrix.
$\bm y_i \in \R^i$ are the coefficients in the basis $\bm V_i$,
determined by solving the $\Ng \times i$ least squares (LS) problem 
$\bm y_i := \min_{\bm y \in \R^i} \| \bm{\delta \rho}_0 - \bm{\mathcal E} (\bm x_0 + \bm V_i \bm y) \|$,
which due to the Arnoldi decomposition
is equivalent to the $(i+1) \times i$ LS problem
\begin{equation}\label{eq:small_ls}
  \bm y_i = \min_{\bm y \in \R^i} \| \beta \bm e_1 - \bm H_i \bm y \|,
\end{equation}
where $\beta = \| \bm r_0 \|$ and $\bm e_1$ is the first column of the
conformable identity matrix. Let the \emph{estimated residual} be $\widetilde{\bm r}_i := \beta \bm e_1 - \bm H_i \bm y_i$ then GMRES terminates when $\|\widetilde{\bm r}_i\|$ is below a tolerance $\tau$ since $\| \bm r_i \| = \|
\widetilde{\bm r}_i \| \leq \tau$.
However, the first equality is only true when \emph{all} the
matrix-vector products $\bm{\mathcal E} \bm v_i$ are computed \emph{exactly}.
In our nested linear problem setting this is not the case as we only solve
the inner SEs to finite tolerance. We make this explicit by denoting
the inexact matrix-vector product as
$\widetilde{\bm{\mathcal E}}^{(i)} \bm v_i \approx \bm{\mathcal E} \bm v_i$ at the $i$-th GMRES iteration.
As a consequence only the \emph{inexact} Arnoldi decomposition
\begin{equation}\label{eq:inexact_arnoldi_decomposition}
    \big[ \widetilde{\bm{\mathcal E}}^{(1)} \bm v_1, \widetilde{\bm{\mathcal E}}^{(2)} \bm v_2, \cdots, \widetilde{\bm{\mathcal E}}^{(i)} \bm v_i \big] = \bm V_{i+1} \bm H_i
\end{equation}
holds during the GMRES, such that $\| \bm r_i \| \neq \| \widetilde{\bm r}_i \|$.
The usual termination criterion $\| \widetilde{\bm r}_i \| \leq \tau$ may thus \emph{not} guarantee $\| \bm r_i \| \leq \tau$.

\begin{figure}
    \begin{center}
    \begin{overpic}[width=0.9\textwidth]{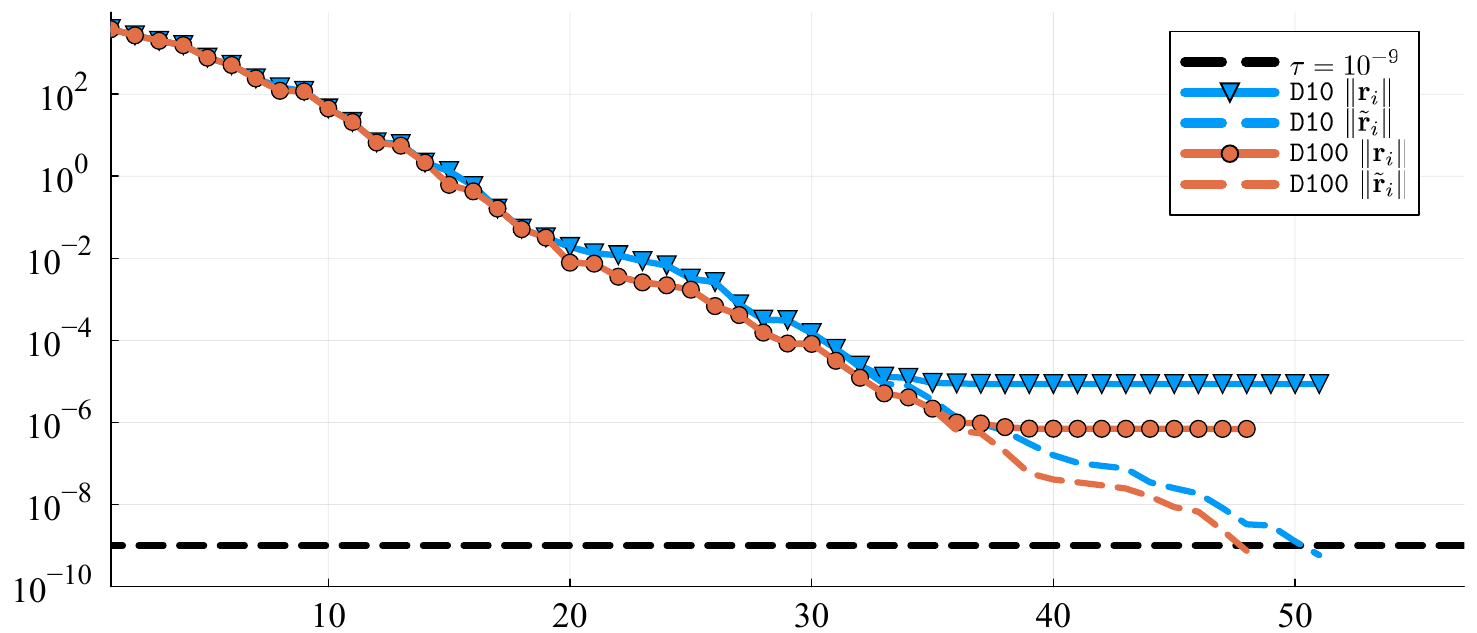}
    \put(35,-2.5){GMRES iteration number $i$}
    \put(-2,5){\rotatebox{90}{true or estimated residuals}}
  \end{overpic}
    \end{center}
  \vspace*{2mm}
  \caption{True residuals $\|\bm r_i\|$ (solid lines) and estimated
  residuals $\|\tilde{\bm r}_i\|$ (dashed lines) v.s. GMRES
  iteration for system in Section~\ref{sec:structure_dyson}. The baseline strategies
  $\taucg_{i,n} = \tau/10,\ \tau/100$ (\sdten, \sdhun) are used for the CG
  solvers for the SEs.}
  \label{fig:true_est_res_diff}
\end{figure}
This is illustrated in Figure~\ref{fig:true_est_res_diff}
for solving the Dyson equation (with no preconditioner)
for an aluminium bulk system with $40$ atoms,
see Section~\ref{sec:Al} for details of the experimental setup.
We consider two baseline strategies \sdten and \sdhun,
in which we set the tolerances of the CG solvers for all SEs
to be the desired GMRES accuracy $\tau$
divided by a constant factor $10$ or $100$, respectively.
In the notation of Section~\ref{sec:structure_dyson} we thus set
$\taucg_{i,n} = \tau/10$ or $\taucg_{i,n} = \tau/100$.
Both the estimated norms $\| \widetilde{\bm r}_i \|$
(accessible during the GMRES procedure) and the true residual norms $\| \bm r_i \|$
(computed explicitly by using tight SE tolerances $\taucg_{i,n} = 10^{-16}$)
are shown in Figure~\ref{fig:true_est_res_diff} for $\tau = 10^{-9}$.
Initially the estimated residuals match the true residuals.
However, around iteration 35 we observe a deviation of both quantities:
while the estimated residuals keep decreasing, the true residuals stagnate.
At convergence, i.e.~when $\| \widetilde{\bm r}_i \| \leq \tau$,
the returned solution turns out to be at least $3$ orders of magnitude
less accurate than $\tau$.

In this work we will develop adaptive strategies for choosing the CG tolerances $\taucg_{i,n}$,
which will not only cure this deficiency, but also bring computational savings
compared to choosing $\taucg_{i,n}$ statically across all GMRES iterations.

\subsection{Theoretical analysis of the inexact GMRES method for the Dyson equation}\label{sec:inexact_gmres}
Krylov subspace methods with inexact matrix-vector products have been studied
in the literature, see~\cite{simoncini2003theory,jasper2004inexact}.
A key result is the following lemma,
which provides a bound of the true residual $\| \bm r_m \|$
by the estimated residual $\| \widetilde{\bm r}_m \|$
and the error in the inexact matrix-vector products
$\big(\widetilde{\bm{\mathcal E}}^{(i)} - \bm{\mathcal E}\big) \bm v_i$ for $i=1,\ldots, m$.

\begin{lemma} \label{lem:inexact_residual_norm}
  Given the notations developed above, we have
  \begin{equation}\label{eq:inexact_residual_norm}
      \| \bm r_m \| \leq \| \bm E_0 \bm x_0 \| + \| \widetilde{\bm r}_m \| + \sum_{i=1}^m \| \bm E_i \bm v_i \| \ \big| \left[ \bm y_m \right]_i \big|,
  \end{equation}
  where $\bm E_i := \widetilde{\bm{\mathcal E}}^{(i)} - \bm{\mathcal E}$ is the error matrix at the $i$-th GMRES iteration and $\left[ \bm y_m \right]_i$ is the $i$-th element of the vector $\bm y_m$ defined in \eqref{eq:small_ls}. 
\end{lemma}
This result is a restatement of~\cite[Proposition 4.1]{simoncini2003theory} adapted to our notation.
The behaviour of $\left| \left[ \bm y_m \right]_i \right|$ is also well-understood
and can in fact be bounded by the estimated residual of the previous iteration
$\| \widetilde{\bm r}_{i-1} \|$:
\begin{lemma}[\cite{simoncini2003theory}, Lemma 5.1] \label{lem:gmres_y_k_bound}
  Given $\bm y_m$ as in \eqref{eq:small_ls}, we have
    \begin{equation} \label{eq:gmres_y_k_bound}
        \left| \left[ \bm y_m \right]_i \right| \leq \dfrac{1}{\sigma_{m}(\bm H_m)} \| \widetilde{\bm r}_{i-1} \|, \quad \text{for } i=1,2,\cdots, m,
    \end{equation}
    where $\sigma_{m}(\cdot)$ is the $m$-th singular value of a matrix sorted
    in descending order.
\end{lemma}
For completeness we provide proofs of both results in \supplement \ref{supp:proofs}.
Notably a simpler proof of Lemma~\ref{lem:gmres_y_k_bound} compared to the original one in~\cite{simoncini2003theory} is given.
Combining Lemmas \ref{lem:inexact_residual_norm} and \ref{lem:gmres_y_k_bound}
immediately yields the following theorem.
\begin{theorem}\label{thm:inexact_gmres}
    For any tolerance $\tau >0$,
    if the error matrices $\bm E_0, \cdots, \bm E_m$ satisfy
    \begin{equation}\label{eq:inexact_gmres_bound}
        \| \bm E_0 \bm x_0 \| \leq c_0\tau,\quad \|\bm E_i \bm v_i\| \leq c_i \dfrac{\sigma_m(\bm H_m)}{\|\widetilde{\bm r}_{i-1}\|}\tau, \quad \text{for} \quad i=1,2,\cdots, m,
    \end{equation}
    and the estimated residual norm is such that $\| \widetilde{\bm r}_m \|\leq c_{m+1}\tau$, for some constants $c_0, \cdots, c_{m+1} > 0$ with
  $\sum_{i=0}^{m+1} c_i \leq 1$, then the method terminates at the $m$-th iteration with the accuracy $\tau$, i.e., $\|\bm r_m\| \leq \tau$.
\end{theorem}

The estimated residual norm $\| \widetilde{\bm r}_i \|$ is monotonically decreasing.
Therefore, Theorem~\ref{thm:inexact_gmres} suggests that we can be
\emph{increasingly crude} when computing the approximate matrix-vector products
$\widetilde{\bm{\mathcal E}}^{(i)} \bm v_i$ as the iteration proceeds.
As a result larger CG tolerances $\taucg_{i,n}$ can be employed,
thereby saving computational cost.
Based on Theorem~\ref{thm:inexact_gmres}
we propose the following inexact GMRES algorithm.

\begin{algorithm}
  \caption{Inexact GMRES($m$)}\label{alg:inexact_gmres}
  \begin{algorithmic}[1]
  \Require $\bm{\mathcal E} \in \C^{\Ng\times \Ng}$, $\bm b \in \C^{\Ng}$, $\bm x_0 \in \C^{\Ng}$, restart size $m \in \N$, tolerance $\tau > 0$; initial guess $s$ of $\sigma_m(\bm H_m)$, a procedure that given any error $\tau' > 0$ and a vector $\bm v \in \C^{\Ng}$ returns an inexact matrix-vector product $\widetilde{\bm{\mathcal E}} \bm v$ such that $\| ( \bm{\mathcal E} - \widetilde{\bm{\mathcal E}} ) \bm v \| \leq \tau'$.
  \Ensure An approximate solution $\bm x$ with $\| \bm b - \bm{\mathcal E} \bm x \| \leq \tau$
  \State Compute $\widetilde{\bm r}_0 = \bm b - \widetilde{\bm{\mathcal E}}^{(0)} \bm x_0$ with $\| (\bm{\mathcal E} - \widetilde{\bm{\mathcal E}}^{(0)}) \bm x_0 \| \leq \tau/3$ \label{line:inexact_GMRES_1}
  \State $\beta = \| \widetilde{\bm r}_0 \|$, $\bm v_1 = \widetilde{\bm r}_0 / \beta$, $\bm V_1 = \bm v_1$, $\bm H_0 = [\ ]$ %
  \For{$i = 1,2,\ldots, m$}
  \State Compute $\bm w = \widetilde{\bm{\mathcal E}}^{(i)} \bm v_i$ with $\| (\bm{\mathcal E} - \widetilde{\bm{\mathcal E}}^{(i)}) \bm v_i \| \leq \dfrac{s}{3m} \dfrac{1}{\|\widetilde{\bm r}_{i-1}\|}\tau$ \label{line:tau}
  \State Build $\bm V_{i+1}$, $\bm H_i$ by Arnoldi process, compute the estimated residual $\| \widetilde{\bm r}_i \|$
  \If{$\| \widetilde{\bm r}_i \| \leq \tau / 3$} \label{line:inexact_GMRES_termination}
  \State Compute $\sigma_i(\bm H_i)$; Compute $\bm y_i$, $\bm x = \bm x_0 + \bm V_i \bm y_i$
  \IIf{$s \leq \sigma_i(\bm H_i)$} \Return{$\bm x$}\EndIIf
  \IIf{$s > \sigma_i(\bm H_i)$} Restart: $s = \sigma_i(\bm H_i)$, $\bm x_0 = \bm x$, go to Line~\ref{line:inexact_GMRES_1} \label{line:inexact_GMRES_s_update} \EndIIf
  \EndIf
  \EndFor
  \If{$\| \widetilde{\bm r}_m \| > \tau / 3$}
  \State Compute $\sigma_m(\bm H_m)$; Compute $\bm y_m$, $\bm x = \bm x_0 + \bm V_m \bm y_m$
  \State Restart: $s = \sigma_m(\bm H_m)$, $\bm x_0 = \bm x$, go to Line~\ref{line:inexact_GMRES_1} \label{line:inexact_GMRES_restart}
  \EndIf
  \end{algorithmic}
\end{algorithm}
Compared to the original work~\cite{simoncini2003theory}
our version of inexact GMRES employs an adaptive strategy
to estimate the $m$-th singular value $\sigma_m(\bm H_m)$ of the Hessenberg matrix.
Note that $\sigma_m(\bm H_m)$ in bound \eqref{eq:inexact_gmres_bound}
is only available \emph{after} the $m$-th iteration
and thus cannot be computed in practice.
Here, we propose to introduce a variable $s$,
which %
gets updated along with the GMRES iterations,
such that at convergence
$s$ is ensured to be a lower bound to $\sigma_m(\bm H_m)$.
Concretely we change $s$ under two circumstances.
First, during a normal restart of the GMRES after $m$ iterations.
Then we set $s$ to the smallest singular value so far encountered,
$s = \sigma_m(\bm H_m)$ (line~\ref{line:inexact_GMRES_restart}).
Second, if convergence of GMRES is flagged by the estimated residual norm
(line~\ref{line:inexact_GMRES_termination}),
but $s$ happens to overestimate $\sigma_i(\bm H_i)$
--- implying that the assumptions of Theorem~\ref{thm:inexact_gmres} are violated.
In such a case we perform an additional restart, updating $s$ again
to the smallest singular value of the Hessenberg matrix,
see line \ref{line:inexact_GMRES_s_update}.
After such a restart the GMRES
convergence criterion in line~\ref{line:inexact_GMRES_termination}
is quickly re-achieved (usually just one or two GMRES iterations).
In practice we find further the decrease of $\sigma_i(\bm H_i)$
w.r.t.~$i$ to be stable across restarts.
Since the restart has updated $s$ we therefore usually have $s \leq \sigma_i(\bm H_i)$
when convergence is re-achieved and  the algorithm terminates.
While such additional restarts are rare,
they guarantee that \eqref{eq:inexact_gmres_bound}
and thus that $\| \bm r_m \| < \tau$ are satisfied
--- such that we obtain the prescribed accuracy
despite employing inexact matrix-vector products.

This strategy for handling the unknown singular values
$\sigma_m(\bm H_m)$ in the context of inexact GMRES appears to be novel,
despite being a natural refinement of well-established ideas in inexact GMRES.

\begin{remark}\label{rem:other_distributions}
  For simplicity, in Algorithm~\ref{alg:inexact_gmres} we chose an equal 
  distribution of error between the three 
  terms of \eqref{eq:inexact_residual_norm}, i.e., 
   $c_0 = c_{m+1} = 1/3$ and 
   $c_1 = \cdots = c_{m} = 1/(3m)$ in Theorem~\ref{thm:inexact_gmres}.
\end{remark}

The missing ingredient to apply this inexact GMRES procedure
to the Dyson equation is a recipe for choosing the CG tolerances $\taucg_{i,n}$
in a way that the error
$\| (\bm{\mathcal E} - \widetilde{\bm{\mathcal E}}^{(i)}) \bm v_i \|$
stays below the bound in line~\ref{line:tau} of Algorithm~\ref{alg:inexact_gmres}.
First note that the inexact matrix-vector product $\widetilde{\bm{\mathcal E}}^{(i)} \bm v_i$
is computed by
\begin{equation}\label{eq:inexact_op_discretised}
  \begin{aligned}
    &\bm{\mathcal E} \bm v_i \approx \bm{\widetilde{\mathcal E}}^{(i)} \bm v_i \\
     = & \bm v_i - \|\bm K \bm v_i\| \sum_{n=1}^{\Nocc} 2 f_n \Real \Big((\bm F^{-1} \bm{\phi}_n)^* \odot ( \bm F^{-1} \bm{\delta \phi^P}_n + \bm F^{-1} \widetilde{\bm{\delta \phi^Q}_n}^{(i)} ) \Big) + \delta f_n \left| \bm F^{-1} \bm{\phi}_n \right|^2,
  \end{aligned}
\end{equation}
that is \eqref{eq:op_discretised}, where instead of the exact SE solutions
$\bm{\delta \phi^Q}_n$ one employs the approximate orbital responses
$\widetilde{\bm{\delta \phi^Q}_n}^{(i)}$.
Introducing the shorthands
\begin{equation}\label{eq:discretised_sternheimer}
  \underbrace{\bm Q \left( \HrhoB - \epsilon_n \bm I_{\Nb} \right) \bm Q}_{=: \bm A_n} \bm{\delta \phi^Q}_n = \underbrace{- \bm Q \bm F\left( (\bm K \bm v_i) \odot \left( \bm F^{-1} \bm{\phi}_n \right) \right)/\|\bm K \bm v_i\|}_{=: \bm b_n^{(i)}}
\end{equation}
for \eqref{eq:dicretized_sternheimer_equation}
the relationship between the approximate orbital responses
and the CG tolerances $\taucg_{i,n}$ is thus
$\|\bm b_n^{(i)} - \bm A_n \widetilde{\bm{\delta \phi^Q}_n}^{(i)}\| \leq \taucg_{i,n}$.
With this notation the following lemma provides the desired bound.
\begin{lemma}\label{lem:cgtol_E}
  The error of the application of the inexact dielectric adjoint $\widetilde{\bm{\mathcal E}}^{(i)}$ on a vector $\bm v_i \in \R^{\Ng}$ is bounded by
  \begin{equation}
      \| (\bm{\mathcal E} - \widetilde{\bm{\mathcal E}}^{(i)}) \bm v_i \| 
      \leq 2 \left\| \bm K \bm v_i\right\| \|\bm{F}^{-1}\bm{\Phi}\|_{2,\infty}  \dfrac{\sqrt{\Ng\Nocc}}{\sqrt{|\Omega|}} \max_{n=1,\cdots,\Nocc} \dfrac{f_n}{\epsilon_{\Nocc + 1} - \epsilon_n} \taucg_{i,n},
  \end{equation}
  where $\|\bm M\|_{2,\infty} := \max\limits_{i = 1, \cdots, n} \| \bm M_{i,:} \|_2$ is the maximum row norm\footnote{See e.g.~\cite{lewis2023listpopularinducedmatrix} for the justification of this notation.} for $\bm M \in \C^{n\times m}$.
\end{lemma}
\begin{proof}[Proof of Lemma~\ref{lem:cgtol_E}]
    Recall that $\bm \Phi \in \C^{\Nb \times \Nocc}$ contains the eigenvectors of $\HrhoB$ with eigenvalues $\epsilon_1, \cdots, \epsilon_{\Nocc}$. Let $\bm \Phi_{\perp} \in \C^{\Nb \times (\Nb - \Nocc)}$ be the orthonormal matrix whose columns are the eigenvectors of $\HrhoB$ with eigenvalues $\epsilon_{\Nocc + 1}, \cdots, \epsilon_{\Nb}$.
    Then we have $\bm Q = \bm \Phi_{\perp} \bm \Phi_{\perp}^\mathsf{H}$. Let $\bm D_n = \diag(\epsilon_{\Nocc + 1} - \epsilon_n, \cdots, \epsilon_{\Nb} - \epsilon_n) \in \R^{(\Nb - \Nocc) \times (\Nb - \Nocc)}$ be a diagonal matrix. Then direct manipulation yields $\bm A_n = \bm \Phi_{\perp} \bm D_n \bm \Phi_{\perp}^\mathsf{H}$, and thus $\bm A_n^{\dagger} = \bm \Phi_{\perp} \bm D_n^{-1} \bm \Phi_{\perp}^\mathsf{H}$, where $\bm A_n^{\dagger}$ denotes the pseudo-inverse of $\bm A_n$ in \eqref{eq:discretised_sternheimer}. 
    Note, that this requires a gap, i.e.~$\epsilon_{\Nocc + 1} - \epsilon_{\Nocc} > 0$,
    which is always assured: the definition of $\Nocc$ implies
    $f(\epsilon_{\Nocc + 1}) < f(\epsilon_{\Nocc})$,
    which in turn by the strict monotonicity of $f$
    implies that $\epsilon_{\Nocc + 1}> \epsilon_{\Nocc}$.
    Hence we have $\bm A_n^{\dagger} \bm A_n = \bm A_n \bm A_n^{\dagger} = \bm Q$. An exact solution of \eqref{eq:discretised_sternheimer} is given by ${\bm{\delta \phi^Q}_n}^{(i)} = \bm A_n^{\dagger} \bm b_n^{(i)} \in \range(\bm Q)$. 

    Since $\bm b_n^{(i)} \in \range(\bm Q)$, therefore the Krylov subspace $\mathcal{K}_k(\bm A_n, \bm b_n^{(i)})$ is contained in $\range(\bm Q)$, and therefore the approximate solution $\widetilde{\bm{\delta \phi^Q}_n}^{(i)}$ produced by the CG algorithm
applied to \eqref{eq:discretised_sternheimer} is also in $\range(\bm Q)$. We have
    \begin{equation}
          \bm{z}_{n}^{(i)} = {\bm{\delta \phi^Q}_n}^{(i)} - \widetilde{\bm{\delta \phi^Q}_n}^{(i)}  = \bm A_n^{\dagger} \Big(\bm b_n^{(i)} - \bm A_n \widetilde{\bm{\delta \phi^Q}_n}^{(i)}\Big).
    \end{equation}
    Furthermore,
    \begin{equation}\label{eq:z_bound}
        \| \bm{z}_{n}^{(i)} \| \leq \| \bm A_n^{\dagger} \| \| \bm b_n^{(i)} - \bm A_n \widetilde{\bm{\delta \phi^Q}_n}^{(i)} \| \leq \dfrac{1}{\epsilon_{\Nocc+1} - \epsilon_n} \taucg_{i,n}.
    \end{equation}
    Combine \eqref{eq:op_discretised}, \eqref{eq:inexact_op_discretised} and
    defining $\bm Y := [ f_1 \bm z_1^{(i)}, \cdots, f_{\Nocc} \bm z_{\Nocc}^{(i)} ] \in
    \C^{\Nb \times \Nocc}$, we have
    \begin{equation}\label{eq:dielectricdiff}
      \begin{aligned}
          \| (\bm{\mathcal E} - \widetilde{\bm{\mathcal E}}^{(i)}) \bm{v}_i\| &= 2 \left\| \bm K \bm{v}_i\right\| \left\| \sum_{n=1}^{\Nocc} f_n \Real \left( \left(\bm F^{-1} \bm{\phi}_n\right)^* \odot \left( \bm F^{-1} \bm{z}_n^{(i)} \right) \right) \right\|                                                             \\
          & = 2 \left\| \bm K \bm{v}_i\right\| \left\| \Real\left((\bm{F}^{-1}\bm{\Phi})^* \odot (\bm{F}^{-1}\bm{Y})\right) \cdot \bm{1}_{\Nocc} \right\| \\
          & \leq 2 \left\| \bm K \bm{v}_i\right\| \left\| (\bm{F}^{-1}\bm{\Phi})^* \odot (\bm{F}^{-1}\bm{Y}) \right\|_{2} \, \|\bm{1}_{\Nocc}\| \\
          & \leq 2 \left\| \, \bm K \bm{v}_i\right\| \, \|\bm{F}^{-1}\bm{\Phi}\|_{2,\infty} \, \|\bm{F}^{-1}\bm{Y}\|_{1,2} \sqrt{\Nocc},
      \end{aligned}
\end{equation}
where $\bm{1}_{\Nocc} = [1, \cdots, 1]^\mathsf{T} \in \R^{\Nocc}$
and we used \cite[Theorem 5.5.3]{roger1991topics}
to bound the norm of the Hadamard product by the product of these two norms.
Furthermore $\|\cdot\|_{1,2}$ is the maximum column norm for a matrix
(in analogy to $\|\cdot\|_{2,\infty}$).
The result then follows from
\begin{equation}
    \left\|\bm{F}^{-1}\bm{Y}\right\|_{1,2}
    \leq \|\bm{F}^{-1}\| \, \|\bm Y\|_{1,2}
    = w^{-1} \max_{n=1,\cdots,\Nocc} f_n \| \bm{z}_n^{(i)} \|.
\end{equation}
\end{proof}
\begin{remark}\label{rmk:sternheimer_solution}
  We remark here that ${\bm{\delta \phi^Q}_n}^{(i)} = \bm A_n^{\dagger} \bm b_n^{(i)} \in \range(\bm Q)$ is the unique solution to the SE \eqref{eq:discretised_sternheimer} in $\range(\bm Q)$. Since all iterates (as well as the conjugate gradients, residuals) of CG algorithm all lie in $\range(\bm Q)$, ${\bm{\delta \phi^Q}_n}^{(i)}$ is thus also the solution that the CG algorithm converges to in exact arithmetic. In practical implementations, each CG iterate (also the conjugate gradient, residual) is projected onto $\range(\bm Q)$ to prevent the loss of orthogonality due to numerical errors. It is also worth noting that the solution to the SE is not unique in the full space, e.g., ${\bm{\delta \phi^Q}_n}^{(i)} + \bm \Phi \bm \alpha$ for any $\bm \alpha \in \C^{\Nocc}$ is also a solution.
\end{remark}

Combining Lemma~\ref{lem:cgtol_E} with Theorem~\ref{thm:inexact_gmres} we obtain immediately
our main theoretical result.
\begin{theorem}\label{thm:inexact_gmres_dyson}
  Consider the inexact GMRES method applied to the discretised Dyson equation \eqref{eq:dyson_discretised} $\bm{\mathcal E} \bm{\delta \rho} = \bm{\delta \rho}_0$ and define the prefactor
  \begin{equation}
      C_{i,n} = \dfrac{\sqrt{|\Omega|}(\epsilon_{\Nocc + 1} - \epsilon_n)}{2 f_n \left\| \bm K \bm v_i \right\| \left\| \bm{F}^{-1}\bm{\Phi} \right\|_{2,\infty} \sqrt{\Ng\Nocc} }
      \label{eqn:bound_prefactor}
  \end{equation}
  where $\bm v_0 := \bm x_0$ (initial guess).
  For any prescribed accuracy $\tau>0$,
  if the tolerance $\taucg_{i,n}$ of the Sternheimer equation CG solver satisfies
  \begin{equation}
      \taucg_{0,n} \leq C_{0,n} \, \dfrac{\tau}{3}
      \quad\text{and}\quad
      \taucg_{i,n} \leq C_{i,n} \, \dfrac{\sigma_m(\bm H_m)}{3m} \dfrac{\tau}{\|\widetilde{\bm r}_{i-1}\|}
  \end{equation}
  for $i=1,\cdots, m$, $n=1,\cdots, \Nocc$
  and the estimated residual of the GMRES method satisfies
  $\|\widetilde{\bm r}_m\| \leq \tau/3$, then the method terminates at
  the $m$-th iteration with accuracy $\tau$, i.e., true residual satisfies $\|\bm r_m\| \leq
  \tau$.
\end{theorem}

\begin{remark}[Superlinear convergence]\label{rmk:superlinear_convergence}
  We remark here that with the adaptive CG tolerances chosen
  according to Theorem~\ref{thm:inexact_gmres_dyson} we can expect
  the inexact GMRES method applied to the Dyson equation
  to converge \emph{superlinearly} with respect to the total number of CG iterations.
  This is demonstrated qualitatively by the following argument.
  Suppose GMRES converges linearly with respect to the number of 
  GMRES iterations $\|\bm r_i\| \leq K^i \|\bm r_0\|$ for $0<K<1$, 
  then $\taucg_{i,n} \sim 1/\|\widetilde{\bm r}_{i-1}\| \sim 1/K^{i-1}$.
  Since CG itself converges at least linearly with respect to its own iteration count,
  the number of CG iterations required at outer step~$i$ decreases approximately linearly as GMRES progresses.
  Expressed in terms of the total number of CG iterations, the overall convergence
  is therefore faster than the linear rate of GMRES.
\end{remark}

\subsection{Preconditioning the Dyson equation} \label{sec:kerker}

From the arguments in Section~\ref{sec:dfpt}
one can view the Dyson equation as a linearisation of the SCF problem~\eqref{eqn:fixedpoint}
around the fixed-point (ground state) density.
As a result we would expect any technique for improving SCF convergence
to be transferable to accelerating convergence when solving the Dyson equation
--- namely by a mere linearisation of the entire SCF algorithm itself.
For example the GMRES algorithm can be viewed as the linearisation
of Anderson-accelerated fixed-point iterations,
which is indeed one standard technique to solve SCF problems
(also sometimes called Pulay mixing or the DIIS method).
For further details on the relationship
between Anderson acceleration and GMRES
see the discussion in~\cite{Fang2009,Walker2011}.
By a similar line of arguments all preconditioning schemes
which have been used to accelerate the SCF should also be used
within the GMRES when solving the Dyson equation~\cite{Gerhorst2024}
for the same material.
Indeed, it is effectively a large spectral condition number
of the dielectric adjoint operator
$\mathcal{E} := 1 - \chi_0 \KHxc_\rho$
--- a material-dependent quantity ---
which is responsible to both slow SCF and slow Dyson equation convergence.

In the SCF context the Kerker preconditioner~\cite{kerker1981efficient} is widely employed.
It effectively cures the
divergence of the Coulomb operator $K_\rho^H$ for large wavelengths (small $\|G\|$),
which in metallic materials is not compensated
by the independent particle susceptibility operator $\chi_0$.
It thus causes the condition number of $\mathcal{E}$ to grow with system size%
~\cite{herbstBlackboxInhomogeneousPreconditioning2020}.
Based on above observations we employ the Kerker preconditioner $\bm T$
when solving the Dyson equation for large metallic systems,
i.e. we solve the preconditioned Dyson equation
\begin{equation}\label{eq:precond_dyson}
  \text{Seek } \bm{\delta \rho} \in \R^{\Ng} \text{ such that } \bm T (\bm I_{\Ng} - \bm \chi_0 \bm K) \bm{\delta \rho} = \bm T \bm{\delta \rho}_0.
\end{equation}

\paragraph{Discrete form of the Kerker preconditioner}
Denoting the discrete Fourier transform matrix by $\bm W\in\C^{\Ng\times\Ng}$
(see \eqref{eq:fft_ifft}),
the discretised Kerker preconditioner
is $\widetilde{\bm T} = \bm W^{-1} \bm D \bm W$,
where $\bm D \in \C^{\Ng \times \Ng}$, its representation in Fourier space,
is the diagonal matrix with entries
\[
\frac{\|\bm G\|^2}{\|\bm G\|^2 + \alpha^2}, \qquad \bm G\in\G^{\Box},
\]
and $\alpha>0$ is an appropriately chosen parameter.
In practice, to ensure that applying the Kerker preconditioner
conserves the total charge one employs the modification
\begin{equation}\label{eq:kerker}
    \bm T = \Big(\bm 1_{\Ng} - \frac{1}{\Ng} \bm 1_{\Ng} \bm 1_{\Ng}^{\top}\Big)
    \bm W^{-1} \bm D \bm W + \frac{1}{\Ng} \bm 1_{\Ng} \bm 1_{\Ng}^{\top},
\end{equation}
where $\bm 1_{\Ng} \in \C^{\Ng}$ is the vector of all ones.
This yields the following lemma.
\begin{lemma}\label{lem:kerker_properties}
  The Kerker preconditioner $\bm T$ is a Hermitian positive definite matrix
  with the maximum eigenvalue $\lambda_{\max} = 1$ and $\| \bm T \| = 1$.
\end{lemma}

\begin{proof}
From the definition of the discrete Fourier transform matrix we have,
\[
1_{\Ng}^\top \bm W^{-1} = \sqrt{\Ng}\, \bm e_1^\top,\qquad
\bm e_1^\top \bm W = \frac{1}{\sqrt{\Ng}}\, \bm 1_{\Ng}^\top,\qquad
\bm W^{-1} \bm e_1 = \frac{1}{\sqrt{\Ng}}\, \bm 1_{\Ng}.
\]
Hence
\begin{equation}\label{eq:kerker_preconditioner_1}
    \bm 1_{\Ng}\bm 1_{\Ng}^\top \bm W^{-1} \bm D \bm W
    = D_{11}\, \bm 1_{\Ng}\bm 1_{\Ng}^\top = \Ng \bm W^{-1} \bm e_1 \bm e_1^{\top} \bm W
\end{equation}
where $D_{11}$ is the first diagonal entry of $\bm D$.
Inserting~\eqref{eq:kerker_preconditioner_1} into~\eqref{eq:kerker} gives
\[
    \bm T = 
    \bm W^{-1}(\bm D + (1 - D_{11})\, \bm e_1 \bm e_1^{\top}) \bm W,
\]
which shows that~\eqref{eq:kerker} is equivalent to setting
the first diagonal element of $\bm D$ to unity.
Since $\bm W$ is unitary and $\bm D$ is real diagonal with
entries in $(0,1]$, it follows that $\bm T$ is Hermitian positive definite
and in particular $\|\bm T\|=\lambda_{\max}=1$.
\end{proof}

\paragraph{Application to our analysis of inexact GMRES}
As a consequence of Lemma~\ref{lem:kerker_properties},
the Kerker preconditioner has unit norm and thus
$\|\bm T (\bm{\mathcal E} - \widetilde{\bm{\mathcal E}}^{(i)}) \bm v_i \|
\leq \|(\bm{\mathcal E} - \widetilde{\bm{\mathcal E}}^{(i)}) \bm v_i \|.$
Therefore, Theorem~\ref{thm:inexact_gmres_dyson}
remains valid when applied to the preconditioned Dyson equation~\eqref{eq:precond_dyson}.

\subsection{Practical adaptive strategies}
\label{sec:practical}
\begin{table}
  \newcommand{\rhsterm}{$\dfrac{\sigma_m(\bm H_m)}{3m} \dfrac{ \tau}{\|\widetilde{\bm r}_{i-1}\|}$}
  \centering
  \smaller
  \begin{tabular}{crll}
    \toprule
    \multicolumn{3}{l}{Adaptive strategies\quad$\taucg_{i,n} = $}&Approximations \\
    \midrule
    \sgrt     & $ \dfrac{1}{\left\| \bm K \bm v_i \right\| \left\| \Real(\bm{F}^{-1}\bm{\Phi}) \right\|_{2,\infty}}$
                $ \dfrac{\sqrt{|\Omega|}}{2 f_n \sqrt{\Ng \, \Nocc}}$
                & \rhsterm & \ref{sec:approx_real},  \ref{sec:approx_gap}\\[0.2em]
    \sbal     & $ \dfrac{\sqrt{|\Omega|}}{\sqrt{\Nocc}}$
    $ \dfrac{\sqrt{|\Omega|}}{2 f_n \sqrt{\Ng \, \Nocc}}$
    & \rhsterm & \ref{sec:approx_gap}, \ref{sec:approx_kernel}, \ref{sec:approx_orbital}\\[0.2em]
    \sagr     & $1$\hspace{2.5em} & \rhsterm & Set $C_{i,n}=1$ \\[0.4em]
    \bottomrule\\
    \toprule
    \multicolumn{3}{l}{Baseline strategies\quad$\taucg_{i,n} = $} \\
    \midrule
    \sdten     & $ \sfrac{1}{10} $ & $\tau$\\[0.2em]
    \sdhun    & $\sfrac{1}{100} $ & $\tau$\\[0.2em]
    \sdnorm     & $\sfrac{1}{10 \|\bm{\delta \rho}_0\|} $ & $\tau$\\
    \bottomrule
  \end{tabular}
  \caption{Summary of strategies: the CG tolerances $\taucg_{i,n}$ are set as
      the product of a prefactor (first column)
      and a common factor shared by all strategies (second column).
      For the adaptive strategies the third column lists
      the sections discussing the approximations we make
      over the prefactor
      $C_{i,n}$ of Theorem~\ref{thm:inexact_gmres_dyson}.
  }
  \label{tab:strategies}
\end{table}

The theoretical choice for the CG tolerances
suggested by Theorem~\ref{thm:inexact_gmres_dyson}
guarantees the accuracy of the returned solution,
but tends to be overly conservative.
In our numerical studies we will therefore focus on three modifications
of the prefactor $C_{i,n}$ of Theorem~\ref{thm:inexact_gmres_dyson},
see Table~\ref{tab:strategies}.
Ranging from conservative to aggressive we term the resulting strategies
\sgrt (``guaranteed''), \sbal (``balanced'') and \sagr (``aggressive'')
--- with \sbal providing the best compromise
between efficiency and accuracy in our experiments.
A discussion of the approximations underlying these strategies
is provided in the onset.
For completeness Table~\ref{tab:strategies} also lists the definition of three
baseline strategies, %
in which the CG tolerances is constant across all GMRES iterations.
These are modelled after common naive choices to choose the tolerance
of the inner Sternheimer problems and will be employed
in our numerical experiments for comparison.

\subsubsection{Dropping the imaginary part of the orbitals}\label{sec:approx_real}
For the \sgrt strategy in our numerical experiments we replaced
the orbital term $\|\bm{F}^{-1}\bm{\Phi}\|_{2,\infty}$ by its real part
$\|\Real(\bm{F}^{-1}\bm{\Phi})\|_{2,\infty}$.
In settings where an entirely real-valued representation can be chosen
(such as the setting of lattice-periodic orbitals as we introduced so far)
this modification can be made without introducing an approximation.
But even in more general settings the real and imaginary parts of $\bm{F}^{-1}\bm{\Phi}$
are usually of comparable magnitude, so such an approximation leads
to only a mild increase of the prefactor --- at most a factor of $\sqrt{2}$.
In our experiments we find this to not impact the ability of the \sgrt
strategy to reliably yield the desired accuracy in the Dyson equation solution.

\subsubsection{Dropping the eigenvalue gaps \texorpdfstring{$(\epsilon_{\Nocc + 1} - \epsilon_n)$}{epsilonDiff}}
\label{sec:approx_gap}
This term arises as a bound to the spectral norm of the pseudoinverse $\| A_n^\dagger \|$,
which is used in \eqref{eq:z_bound} to bound the $\ell_2$ error $\| \bm{z}_n^{(i)}\|$
of the Sternheimer solution.
However, in practical computations we employ
a Schur complement approach~\cite{cances2023numerical}
when solving \eqref{eq:discretised_sternheimer},
which modifies the effectively employed operator $A_n$
in \eqref{eq:discretised_sternheimer} leading to an increase
of its smallest eigenvalue, thus a reduction of the spectral norm $\| A_n^\dagger \|$.
In practice making the assumption
$\|\bm z_{n}^{(i)} \| \approx \| \bm r_n^{(i)} \|$,
i.e. setting $(\epsilon_{\Nocc + 1} - \epsilon_n) \approx 1$,
thus turns out to be a good approximation.
This is also demonstrated numerically in \supplement \ref{supp:approx_gap}
for our two example systems from Section~\ref{sec:experiments}.
Motivated by this outcome we keep the terminology ``guaranteed'' for our \sgrt
strategy despite making this approximation over Theorem~\ref{thm:inexact_gmres_dyson}.

\subsubsection{Dropping the Hartree-exchange-correlation kernel term \texorpdfstring{$\| \bm{K} \bm{v}_i \|$}{Kvi}}
\label{sec:approx_kernel}
The Hartree-exchange-correlation kernel is
composed of two terms $\KHxc = K^\text{H} + K^\text{xc}$.
In the absence of spontaneous symmetry breaking
the Hartree term usually dominates
and the XC term can be neglected (random phase approximation).
In a plane-wave basis the Hartree term is a diagonal matrix with entries
$\left(\bm K_{\mathrm F}\right)_{\vec{G},\vec{G}}
= \frac{4\pi}{\| \vec{G}\|^2}$ with the exception
$\left(\bm K_{\mathrm F}\right)_{\bm 0 \bm 0} = 0$
(due to the compensating background charge).
We observe the Coulomb divergence:
for small $\|\bm G\|$ this operator is unbounded.
Since the smallest non-zero $\|\bm G\|$ of the cubic Fourier grid $\G^{\Box}$
is $\frac{2\pi}{\| \bm a_1 \|}$
--- with $\bm a_1$ denoting the largest cell vector ---
we obtain that $\|\bm K_{\mathrm F}\| \leq \frac{\| \bm a_1 \|^2}{\pi}$.
This leads to the naive bound
\begin{equation}
    \label{eq:bound_k_naive}
    \| \bm K \bm v_i \| \simeq \| \bm K_{\mathrm H} \bm v_i \|
    \leq \| \bm K_{\mathrm H} \| \| \bm v_i \|
    = \| \bm K_{\mathrm F} \| \| \bm v_i \|
    \leq \frac{\|\bm a_1\|^2}{\pi} \| \bm v_i \|,
\end{equation}
where in the third step  we used that the real-space representation
of the Hartree kernel is $\bm K_{\mathrm H} = \bm W \bm K_{\mathrm F} \bm W^{-1}$
and that Fourier transforms are unitary.
From this development it seems there is little hope
that the term $\| \bm K \bm v_i \|$ can be dropped:
as the system size (thus $\| \bm a_1 \|$) increases
it may in principle grow quadratically.
This behaviour is indeed observe in practice
as \supplement \ref{supp:approx_kernel} demonstrates
for an aluminium supercell in the absence of preconditioning.

However, when solving the Dyson equation using Kerker-preconditioned GMRES
all Krylov vectors in the GMRES subspace as well as the iterates $\bm v_i$
benefit from the effect of preconditioning.
As the Kerker preconditioner is constructed to counteract the Coulomb
divergence the entries of $\bm v_i$ decay sufficiently
rapidly with decreasing $\|\bm G\|$,
such that $\| \bm K \bm v_i \|$ stays roughly constant as system size grows,
see \supplement \ref{supp:approx_kernel}.
When modelling insulators Kerker preconditioning is not necessary
for $\| \bm K_{\mathrm H} \bm v_i \|$ to remain constant
as the Coulomb divergence is in this case
compensated by the large-wavelength limit of
$\chi_0$~\cite{herbstBlackboxInhomogeneousPreconditioning2020,cances2010dielectric}.
Overall employing appropriate preconditioning of the Dyson equation
the term $\| \bm K \bm v_i \|$ thus stays small even for large systems
motivating us to drop this term for our \sbal and \sagr strategies.

\subsubsection{Replacing the orbital term by its lower bound}
\label{sec:approx_orbital}
For the orbital term $\left\| \bm{F}^{-1}\bm{\Phi} \right\|_{2,\infty}$ we can derive
upper and lower bounds.
\begin{proposition}\label{prop:FinvPhi_bound}
  With the notations defined above, we have
\begin{equation}
    \sqrt{\frac{\Nocc}{|\Omega|}} \leq \left\| \bm{F}^{-1}\bm{\Phi} \right\|_{2,\infty} \leq \sqrt{\frac{\Ng}{|\Omega|}}.
\end{equation}
\end{proposition}
\begin{proof}
  It suffices to prove
  $\sqrt{\frac{\Nocc}{\Ng}} \leq \left\| \bm\Phi_{\mathrm{real}} \right\| \leq 1$
  where $\bm \Phi_{\mathrm{real}} = \sqrt{\frac{|\Omega|}{\Ng}} \, \bm{F}^{-1}\bm{\Phi}$.
  By the definition of $\bm{F}^{-1}$ \eqref{eq:fft_ifft} and the fact that $\bm
  \Phi$ has orthonormal columns, $\bm \Phi_{\mathrm{real}}$ also has
  orthonormal columns. Therefore, for the upper bound, it follows from
  \linebreak
  $\|\bm{\Phi}_{\mathrm{real}}(i,:)\| \leq \|\bm{\Phi}_{\mathrm{real}}\| = 1$
  for $i=1,\cdots, \Ng$. For the lower bound, it follows from 
  \begin{equation}
    \Nocc = \sum_{j=1}^{\Nocc} \|\bm{\Phi}_{\mathrm{real}}(:,j)\|^2 = \sum_{i=1}^{\Ng} \|\bm{\Phi}_{\mathrm{real}}(i,:)\|^2 \leq \Ng \max_{i=1,\cdots, \Ng} \|\bm{\Phi}_{\mathrm{real}}(i,:)\|^2.
  \end{equation}
\end{proof}
The upper bound is not useful since it is very pessimistic,
essentially assuming localisation of the orbitals at a single grid point.
This contrasts with the typical behavior of orbitals in practical
Kohn-Sham DFT calculations, where the (pseudopotential) model is precisely
chosen to achieve a smooth real-space functional form,
resulting in fast Fourier decay.
For this reason we will replace the orbital term
by its lower bound  in our \sbal and \sagr strategies.
However, if this approximation was not satisfactory
we could keep the term $\left\| \bm{F}^{-1}\bm{\Phi} \right\|_{2,\infty}$
as its computation is a cheap one-time cost.

\begin{remark}[Scaling of the bounds with system size]\label{rmk:scaling}
While it is in general easy to construct practical heuristics
for convergence tolerances which work well in small material systems,
it is much harder to make sure these heuristics keep working
for larger, more realistic systems.
To ensure our strategies (see Table~\ref{tab:strategies})
do not deteriorate for larger systems, we will explicitly analyse
their scaling with system size.
Consider extending the simulation cell
by building a supercell made up of $\ell$ replica in one Cartesian direction.
In this process both the cell volume $|\Omega|$
and the number of electrons $\Nocc$ scale linearly with $\ell$.
The approximate scaling of $\Ng$ in equation \eqref{eq:Nb_Ng_definition}
establishes an approximate linear scaling of $\Ng$ with $\ell$ as well.
Consider the exact prefactor of Theorem~\ref{thm:inexact_gmres_dyson}
\begin{align*}
    C_{i,n} &=
    \frac{\epsilon_{\Nocc+1} - \epsilon_n}{%
        \left\| \bm K \bm v_i \right\|
        \left\| \bm{F}^{-1}\bm{\Phi} \right\|_{2,\infty}}
    \frac{\sqrt{|\Omega|}}{2 f_n \sqrt{\Ng \, \Nocc}}.
\end{align*}
Since both the upper and lower bound
of $\left\| \bm{F}^{-1}\bm{\Phi} \right\|_{2,\infty}$
obtained in Proposition~\ref{prop:FinvPhi_bound} do not scale with $\ell$
we can deduce this term to be asymptotically invariant to $\ell$ as well.
Following our arguments in Sections~\ref{sec:approx_gap}~and~\ref{sec:approx_kernel}
an appropriate preconditioning of the CG as well as Dyson equations
should keep the eigenvalue difference as well as the kernel
term invariant to $\ell$.
Overall we expect $C_{i,n} = O(1 / \sqrt{\ell})$.

By similar arguments
we can deduce this scaling
with $\ell$ to be conserved for our proposed \sgrt and \sbal strategies,
but violated for \sagr.
When solving the Dyson equation using the \sagr strategy
we would thus expect the final true GMRES residual to deviate from
the desired tolerance $\tau$ more and more as $\ell$ gets larger.
This trend is indeed observed in our numerical experiments
(see Table~\ref{tab:summary_Al_res}).
\end{remark}

\section{Numerical experiments}
\label{sec:experiments}
\subsection{From theory to computation: methodology and general setup}
\subsubsection{Modelling perfect crystals}
\label{sec:crystals}
In our numerical experiments we will target problems
from condensed matter physics to demonstrate the efficiency and accuracy
of our adaptive CG tolerance strategies.
In condensed matter simulations
materials are commonly modelled as perfect crystals,
i.e.~systems in which a finite-sized simulation cell (the unit cell $\Omega$)
is infinitely repeated periodically, forming a Bravais lattice $\La$.
To avoid dealing with an infinitely large problem,
practical simulations commonly employ the \textit{supercell} approach,
in which the unit cell $\Omega$ is only repeated finite times in each direction.
For example, a $\ell_1 \times \ell_2 \times \ell_3$ supercell is formed
by repeating $\Omega$ respectively $\ell_1$, $\ell_2$, and $\ell_3$
times in the directions of the lattice vectors,
thus leading to a system consisting of  $L = \ell_1 \ell_2 \ell_3$ unit cells.
At the boundary one employs periodic boundary conditions %
to avoid artificial surface effects.

When solving the Kohn-Sham problem on a $\ell_1 \times \ell_2 \times \ell_3$ supercell
we are thus in the setting of Section \ref{sec:mathematical} with respectively
a larger simulation cell.
This will lead to a $L$-times larger eigenvalue problem with $L$-times
more occupied orbitals to be computed
(since $\Nb$, $\Nocc \sim L |\Omega|$ as discussed in Remark~\ref{rmk:scaling}).
When treated naively this would drive up computational cost by a factor $L^3$.
However, assuming no spontaneous symmetry breaking
the KS ground state keeps the $\La$-translational invariance of the Hamiltonian
even when performing the simulation in a supercell.
As a result Bloch theory enables to relabel the supercell orbitals $\phi_n$ as
$\phi_{n\bm k} (\bm r) = \e^{\im \bm k \cdot \bm r} u_{n \bm k}(\bm r)$,
where we re-indexed $n$ as $(n, \bm k)$.
See e.g.~\cite[Section XIII.16]{reed1978methods4} for more details.
Notably  $u_{n \bm k}$ is again $\La$-periodic and the
Kohn-Sham equations \eqref{eq:KS} can be rewritten as
\mbox{$H_{\rho,\bm k} u_{n\bm k} = \epsilon_{n\bm k} u_{n\bm k}$} with
increasing eigenvalues $\epsilon_{n\bm k}$,
orthonormal eigenfunctions $u_{n\bm k}$
and the $\bm k$-point Hamiltonian 
$H_{\rho,\bm k} = \frac{1}{2}(-\im \nabla + \bm k)^2 + \Vext(\vec{r}) + \VHxc_\rho(\vec{r})$
(c.f. \eqref{eqn:Hamiltonian}).
The electronic density $\rho$ is rewritten as (c.f. \eqref{eqn:density})
\begin{equation}
    \label{eq:bzdensity}
    \rho(\bm r) = \dfrac{1}{L} \sum_{\bm k \in \mpgrid} \sum_{n=1}^{\infty} f_{n \bm k} |u_{n\bm k}(\bm r)|^2, \ f_{n\bm k} = \fFD\left(\dfrac{\epsilon_{n\bm k} - \epsilonF}{T}\right), \ \dfrac{1}{L}\sum_{\bm k \in \Omega^*}\sum_{n=1}^{\infty} f_{n\bm k} = N.
\end{equation}
Here, $\Omega^*$ is the first Brillouin zone, the unit cell of the dual lattice $\La^*$ of $\La$.
Moreover the grid $\mpgrid = \Omega^* \cap \left(
    \mathbb{Z}\frac{\mathbf{b}_1}{\ell_1} + \mathbb{Z}\frac{\mathbf{b}_1}{\ell_2} + \mathbb{Z}\frac{\mathbf{b}_1}{\ell_3}
\right)$ is known as the Monkhorst-Pack grid~\cite{monkhorst1976special} in the physics literature.
Assuming the absence of spontaneous breaking of translational symmetry in the density
these equations approach the solution of the Kohn-Sham problem for perfect crystals
in the thermodynamic limit $L \to \infty$.
In this limit the sums over $\bm k$ in \eqref{eq:bzdensity} become integrals over the Brillouin zone $\Omega^*$.
An equivalent interpretation of the supercell method is thus as a quadrature of
such integrals over $\Omega^*$ with quadrature nodes $\mpgrid$.
For more details on the supercell approach see~\cite[Section 2.8]{lin2019math_electronic},
for a proof of the thermodynamic limit in the setting of reduced Hartree-Fock theory,
see~\cite{cances2008localdefects,catto1998mathematical,gontier2016convergencerhf}
and for an analysis of common Brillouin zone quadrature schemes see~\cite{cances2020quadratureBZ}.

This supercell approach thus reduces the complexity massively
as only computations on the unit cell need to be performed.
Moreover each $\bm k$-point-specific Hamiltonian can be treated independently
leading to a natural parallelism, which can be easily exploited during the computation.

Following a similar argument one can employ Bloch theory to obtain a representation
of the application $\chi_0 \delta V$ in terms of Bloch modes.
Comparing to \eqref{eq:chi0} this amounts to adding an additional sum over $\bm k$-points, i.e.
\begin{equation}
    \label{eq:bzchi0}
  \chi_0 \delta V
  = \frac1L \sum_{\bm k \in \mpgrid} \sum_{n=1}^{\Nocck} 2 f_{n\bm k} \Real \big( \phi_{n\bm k}^* \, \delta \phi_{n\bm k} \big) + \delta f_{n\bm k} \, \left| \phi_{n \bm k} \right|^2,
\end{equation}
where we understand $\phi_{n\bm k}(\bm r) = e^{i \bm k \cdot \bm r} u_{n\bm k}(\bm r)$.
In this work we only consider lattice-periodic perturbations $\delta V$,
such that the orbital responses $\delta \phi_{n\bm k} = \delta \phi^P_{n\bm k} + \delta \phi^Q_{n\bm k}$
have identical structure to the expressions in Section \ref{sec:mathematical},
the main change being an additional index $\bm k$, namely
\begin{equation}
    \delta \phi_{n\bm k}^P = \sum_{m=1}^{\Nocck}
    \frac{f_{n\bm k}^2}{f_{n\bm k}^2 + f_{m\bm k}^2} \frac{f_{n\bm k} - f_{m\bm k}}{\epsilon_{n\bm k} - \epsilon_{m\bm k}} \braket{\phi_{m\bm k}}{\delta V \phi_{n\bm k}} \phi_{m\bm k}
\end{equation}
for $n = 1,\ldots,\Nocck$ and $\bm k \in \mpgrid$. $\phi_{n\bm k}^Q$ is again obtained by solving a Sternheimer equation
    $Q \left( H_{\rho,\bm k} - \epsilon_{n\bm k} \right) Q \delta \phi_{n\bm k}^Q = -Q \left(\delta V \phi_{n\bm k}\right)$.
Therefore, Theorem~\ref{thm:inexact_gmres_dyson} and the adaptive strategies proposed
in Section~\ref{sec:practical} remain applicable due to the similar mathematical structure.
The main change is to replace $n$ in the prefactor $C_{i,n}$ by $n\bm k$ 
and note that the number of occupied orbitals $\Nocc$ is now the sum of the number of occupied orbitals
for each $\bm k$-point, i.e. $\Nocc = \sum_{\bm k} \Nocck$.
See \supplement \ref{supp:perfect_crystals} for more details on this extension.

\subsubsection{Implementation and computational setup}
\label{sec:setup}
We implemented the inexact GMRES framework
(Algorithm \ref{alg:inexact_gmres})
within the Density-Functional ToolKit (DFTK)~\cite{herbst2021dftk}
--- an accessible framework for mathematical
research on plane-wave DFT methods written in the \texttt{Julia}
programming language~\cite{Bezanson2017}.
Unless otherwise specified
we followed previous work~\cite{cances2023numerical} on DFPT problems and employed
a uniform Monkhorst-Pack grid to discretise the Brillouin zone,
the PBE exchange-correlation functional~\cite{perdew1996gga} and
GTH pseudopotentials~\cite{goedecker1996separable,hartwigsen1998relativistic}.
Our DFPT test problems are obtained by solving
the self-consistent field problem without any perturbation,
followed by introducing a perturbing potential $\bm{\delta V}_0$
resulting from displacing the atomic positions in the simulation cell.
This perturbation provides the right-hand side $\bm{\chi}_0 \bm{\delta V}_0$
of a Dyson equation \eqref{eq:dyson_discretised},
which is solved using Algorithm \ref{alg:inexact_gmres}.
Throughout our experiments we employ a zero initial guess ($\bm x_0 = \bm 0$)
and use an initial estimate $s=1$ for $\sigma_m(\bm H_m)$.
For the inexact applications %
of the dielectric operator we solve the Sternheimer equations
using the Schur complement trick of~\cite{cances2023numerical} as well as
a CG solver with kinetic energy preconditioning
(the kinetic operator is the scaled inverse Laplacian,
which is diagonal in a Fourier representation).
The employed CG convergence tolerances $\taucg_{i,n\bm k}$
are chosen adaptively according to one of the adaptive strategies
from Table~\ref{tab:strategies}.
For each Sternheimer solve a minimum of one CG iteration is performed
even if the initial residual norm is already below this tolerance.

The full source code to reproduce the experiments,
along with more exhaustive numerical tests, is available
at \url{https://github.com/bonans/inexact_Krylov_response}.
Since version 0.7.14 the adaptive algorithm introduced in this work
is employed as the default Dyson equation solver in DFTK.

\subsubsection{Performance metrics}
In our experiments we will employ the following three metrics
to compare our various strategies:

\paragraph{Cost}
The cost for solving the Dyson equation
is dominated by the cost of solving the $\Nocc$ Sternheimer equations
for each GMRES iteration.
In turn the cost of solving a Sternheimer equation
is directly related to the number of applications of the Hamiltonian $\bm{H}_{\bm\rho}$
as its most expensive step.
We thus employ the total number of Hamiltonian applications $\Nham$
as a measure to evaluate the computational cost. %

\paragraph{Accuracy}
Some strategies for adaptively selecting
the CG tolerances violate theorem~\ref{thm:inexact_gmres_dyson},
such that accuracy of the solutions $\xend$ returned by the inexact GMRES can vary.
We compare accuracies by making reference to the \textit{true} residual
norm $\|\rend\| = \|\bm{b} - \bm{\mathcal E}\xend\|$,
which is obtained by applying the dielectric adjoint $\bm{\mathcal E}$
with the CG tolerance of all Sternheimer equations tightened to $10^{-16}$.
This same residual expression is used
for both preconditioned and non-preconditioned tests.
Similarly when discussing the residuals at the $i$-th GMRES iteration
we make the distinction between the \emph{estimated} residual $\|\widetilde{\bm{r}}_i\|$,
which is employed in the GMRES procedure itself,
and the \emph{true} residuals $\|\bm{r}_i\| = \|\bm{b} - \bm{\mathcal E}\bm{x}_i\|$,
obtained by an explicit re-computation using tight Sternheimer tolerances.

\paragraph{Efficiency}
Different strategies for selecting the CG tolerances show an intrinsic trade-off
between the computational cost and obtained accuracy.
To quantify this balance more systematically we introduce the (absolute) efficiency
\begin{equation}
\eta := -\dfrac{\log_{10} \left( \|\rend\| / \|\bm{r}_0\| \right)}{\Nham} > 0,
\end{equation}
where $\|\bm{r}_0\|$ ($\|\rend\|$) is the initial (final) true residual.
Note that $\eta$ effectively measures the average residual reduction
per Hamiltonian application and can also be interpreted as the
convergence rate in terms of the number of Hamiltonian applications.
Additionally, we use a \emph{relative efficiency} $\releta$
--- that is the ratio of $\eta$ of one strategy with respect to a reference.
Here, one desires to obtain a large value $\releta > 1$
to conclude that the evaluated strategy is \emph{more efficient} than the reference.
For example, $\releta = 2$ indicates that compared to the reference
half as many Hamiltonian applications are needed to achieve the same accuracy.

\subsection{Systematic comparison of strategies on an aluminium supercell}\label{sec:Al}
We first compare our various strategies for setting
the CG tolerances (Table~\ref{tab:strategies}),
both with and without preconditioning in the outer inexact GMRES.
As an example we take a standard problem in condensed matter physics,
namely an elongated metallic supercell.
Specifically we take an aluminium supercell with $40$ atoms
obtained by repeating a $4$-atom cubic unit cell $10$ times in the $x$ direction,
see Figure~\ref{fig:convergence_Al40_new} (A).
On top of the parameters of section \ref{sec:setup} we chose $\Ecut = 40$ Ha,
Fermi-Dirac smearing with temperature $T=10^{-3}$ Ha and a
$1\times 3\times 3$ $\bm k$-point grid. In this setup the Dyson equation has size
$\Ng = 911\,250$ and each application of the dielectric adjoint
$\bm{\mathcal E}$ requires solving $\Nocc = 569$ SEs, each of
size $\Nb \approx 54\,200$.
We remark that in standard computational setups approximations are more consistent
across $\bm k$-points if one constructs the spherical Fourier basis in \eqref{eqn:fourier_bases}
by the condition $\| \bm{G} + \bm{k} \|_2 \leq \sqrt{2 \Ecut}$
--- thus making the basis and thus $\Nb$ differ across $\bm k$ points.

\begin{table}
  \centering
  \smaller
  \begin{tabular}{c|cccc|cccccc}
    \toprule
    Strategy         & \textcolor{blue}{\sagr}    & \textcolor{blue}{\sbal}    & \sdten    & \sdhun    & \textcolor{blue}{\sPagr}    & \textcolor{blue}{\sPbal}   & \textcolor{blue}{\sPgrt}    & \sPdten    & \sPdhun   & \sPdnorm \\
    \midrule
    $\|\rend\|$     & 2e-7        & 6e-9        & 9e-6        & 7e-7        & 2e-8        & \textcolor{orange}{7e-10}  & \textcolor{orange}{2e-10}  & 1e-5        & 1e-6        & \textcolor{orange}{5e-9}   \\
    $\Nham$      & 385\,k      & 453\,k      & 499\,k      & 539\,k      & \textcolor{orange}{168\,k} & \textcolor{orange}{196\,k} & 228\,k      & \textcolor{orange}{194\,k} & 219\,k      & 275\,k      \\
    $\releta$ & 1.53                  & 1.50                  & 1.00                  & 1.04                  & \textcolor{orange}{3.87}             & \textcolor{orange}{3.71}             & \textcolor{orange}{3.31}             & 2.52                  & 2.49                  & 2.47                  \\
    \bottomrule
  \end{tabular}
  \caption{\textbf{Aluminium supercell:} Accuracy of the returned solution
      $\|\rend\|$ (i.e.~true non-preconditioned residual),
  total number of Hamiltonian applications $\Nham$ and relative
  efficiency $\releta$ (referencing \sdten) for the
  aluminium supercell with $\tau = 10^{-9}$. Adaptive strategies are
  highlighted in blue, and the top three strategies for each metric are
  highlighted in orange.}
  \label{tab:summary_Al10_20_-9}
\end{table}

In Table~\ref{tab:summary_Al10_20_-9} we summarise the performance metrics
of various strategies for $\tau = 10^{-9}$ and a GMRES restart size $m=20$.
The CG tolerances
$\taucg_{i,n\bm k}$ are set according to Table~\ref{tab:strategies}
and the relative efficiency $\releta$ is given against the \sdten strategy.
When employing a Kerker preconditioner (Section~\ref{sec:kerker})
in the inexact GMRES we prefix the corresponding strategy
by a \texttt{P}. For instance, \sPgrt refers to employing both the Kerker
preconditioner as well as the \sgrt strategy for selecting CG tolerances.
Note that for \sgrt and \sdnorm we only tested the preconditioned versions
as the non-preconditioned versions are too computationally expensive
without offering additional insights.

We observe our adaptive strategies (highlighted in blue) to consistently yield
higher final accuracy while requiring similar or fewer Hamiltonian applications compared
to baseline strategies. For example, switching from a naive \sdten
or \sdhun to our \sbal strategy speeds up the computation by a factor $1.5$
as illustrated by the relative efficiency row.
By large this is due the ability of our adaptive strategies to linearly
reduce the number of CG iterations to solve the SEs across the GMRES iterations
--- until just $1$ or $2$ CG iterations are needed in the final GMRES steps.
See \supplement \ref{supp:iterations}
for a numerical demonstration of this behaviour.
When combining adaptive CG tolerances with GMRES preconditioning,
this effect augments to saving almost a factor of $4$ compared to \sdten.
This illustrates Kerker preconditioning to bring noteworthy efficiency gains
for metallic systems, in agreement with Section~\ref{sec:kerker}.
With respect to the accuracy we observe that only the preconditioned adaptive
strategies \sPgrt and \sPbal achieve a true residual norm below the desired
tolerance $\tau = 10^{-9}$, while the non-preconditioned \sbal fails in doing so
--- again emphasising the need for appropriate preconditioning
when dropping the Hartree-exchange-correlation kernel term $\| \bm{K} \bm{v}_i \|$
from our estimates (Section~\ref{sec:approx_kernel}).
Furthermore we notice the failure of all \sagr and baseline variants to achieve
the prescribed tolerance.
Overall \sPbal stays in the top $3$ (highlighted in orange) for all metrics
and achieves the desired tolerance in only insignificantly more
Hamiltonian applications compared to \sPdten
--- thus representing the best compromise in this example.

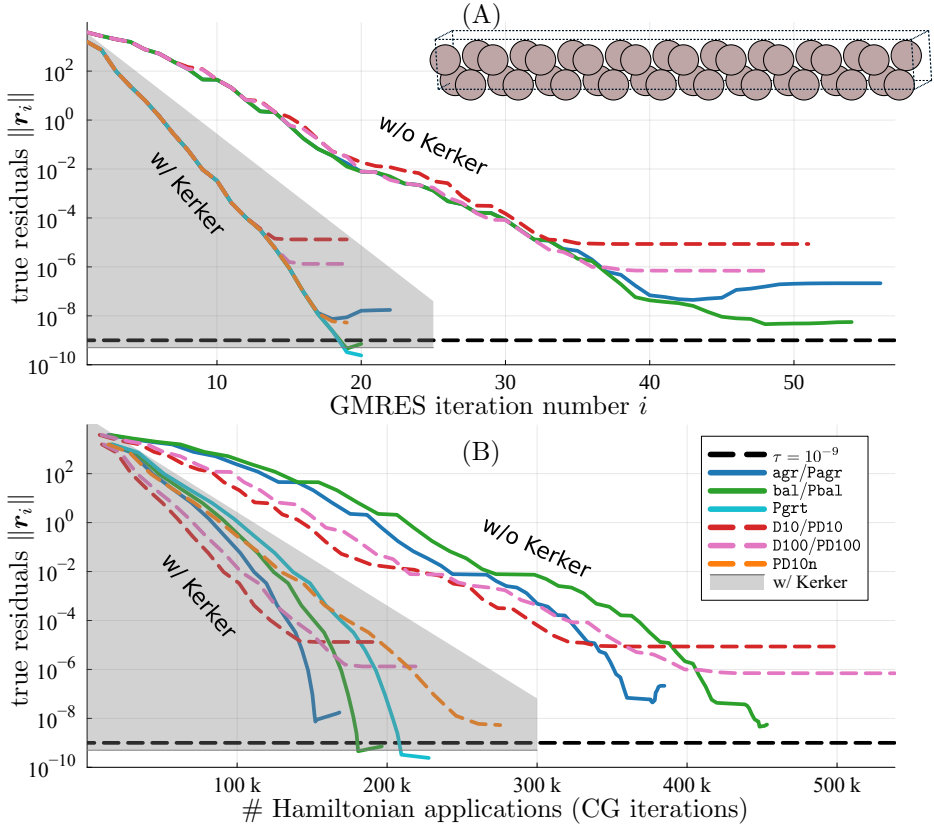
\begin{figure}
    \begin{center}
    \begin{overpic}[width=0.9\textwidth]{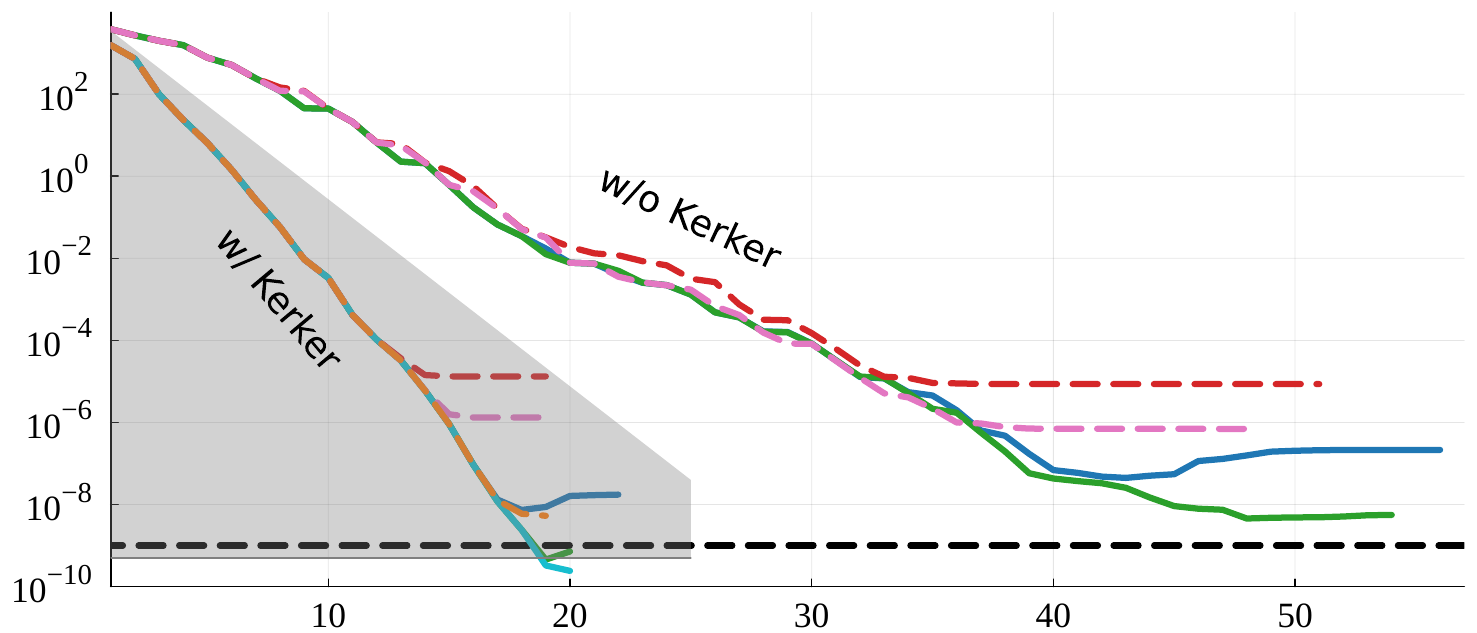}
    \put(45, 33) {\includegraphics[width=7cm]{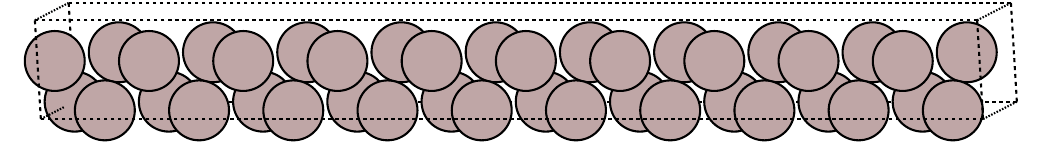}}
    \put(35,-2.5){GMRES iteration number $i$}
    \put(-1.5,10){\rotatebox{90}{true residuals $\| {\bm r}_i \|$}}
    \put(50,42){(A)}
  \end{overpic}
  \\[4mm]
  \begin{overpic}[width=0.9\textwidth]{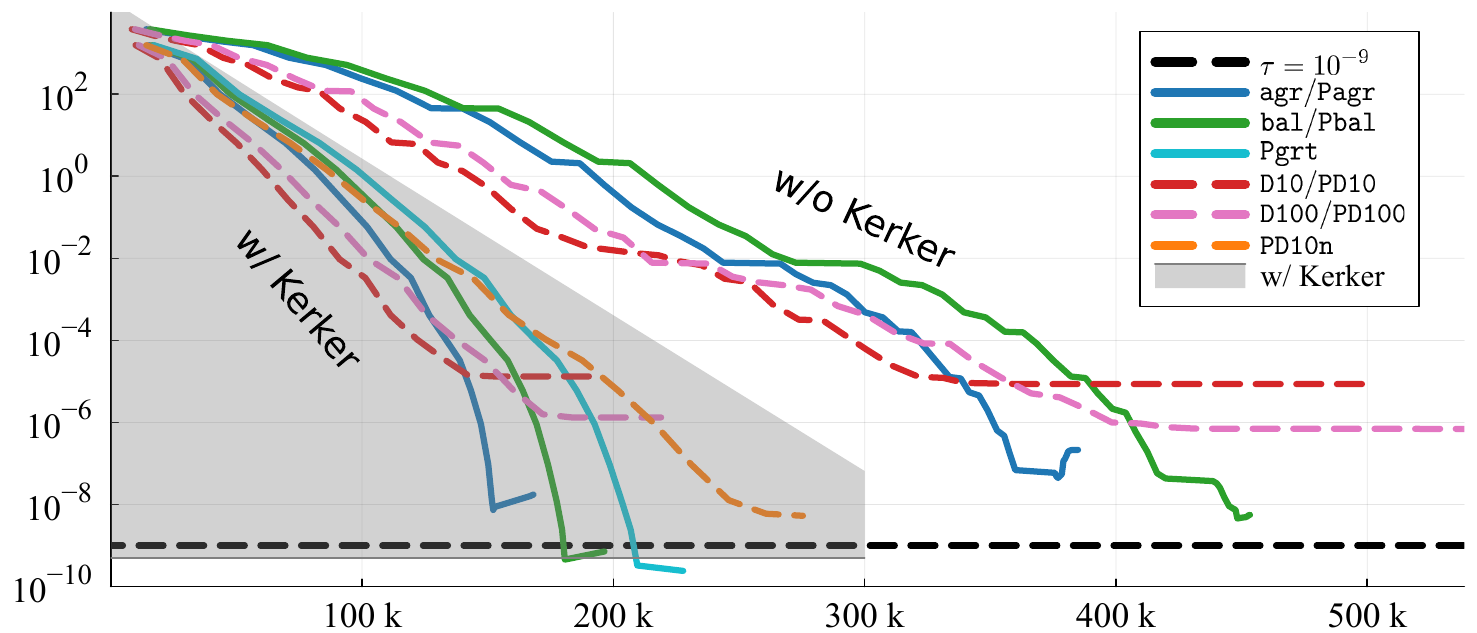}
    \put(25,-2.5){\# Hamiltonian applications (CG iterations)}
    \put(-1.5,10){\rotatebox{90}{true residuals $\| {\bm r}_i \|$}}
    \put(50,38){(B)}
  \end{overpic}
    \end{center}
  \vspace*{2mm}
  \caption{\textbf{Aluminium supercell:}
  True residual norms $\|\bm r_i\|$ v.s. (A) number of GMRES iterations,
  (B) accumulated number of Hamiltonian applications.
  The shaded areas indicate the strategies with Kerker preconditioning.}
  \label{fig:convergence_Al40_new}
\end{figure}

This impression of our proposed adaptive CG tolerance strategies
is confirmed when considering
Figures~\ref{fig:convergence_Al40_new} (A) and (B),
displaying the convergence of the true residuals
of the GMRES method versus the GMRES iteration number $i$
and the accumulated number of Hamiltonian applications.
In particular for our strategies (solid lines) Figure~\ref{fig:convergence_Al40_new} (A)
shows that the convergence rate of the outer GMRES iterations is not
significantly affected by using inexact matrix-vector products $\bm{\mathcal E} \bm v$
until the final stages of the iterations.
Most notably the \sPbal and \sPgrt strategies (i.e. employing appropriate preconditioning)
keep a steady linear convergence wrt.~$i$ until the desired tolerance $\tau$ has been reached.
Finally Figure \ref{fig:convergence_Al40_new} (B) confirms the superlinear convergence
of our adaptive strategies with respect to the number of Hamiltonian applications
and thus the accumulated computational cost
--- as expected from Remark~\ref{rmk:superlinear_convergence}.
Again \sPbal and \sPgrt emerge as the best compromises between accuracy and cost.

We remark that the trends and overall conclusions of this experiment are
stable with respect to target accuracy $\tau$, Krylov size $m$ or the
size of the aluminium system as a more exhaustive study presented
in \supplement \ref{supp:aluminium_parameters} illustrates.

\subsection{Targeting a complex material: Heusler alloy}\label{sec:heusler}
We consider the \linebreak \ch{Fe2MnAl} transition-metal alloy
as an example for a condensed matter system with a considerably more challenging
electronic structure than aluminium.
It is an example of a Heusler alloy, a class of materials,
which are
of practical interest due to their rich and unusual magnetic and electronic properties.
For example, \ch{Fe2MnAl} is a half-metal, which loosely speaking
means that depending on the spin degree of freedom the material behaves
either like a metal or like an insulator.
This behaviour is quite susceptible to the external conditions,
such that the electronic density of the material can change drastically
under a small perturbing potential $\delta V_0$. Numerically this makes
the Dyson equation more challenging to solve. %
See \cite{herbstBlackboxInhomogeneousPreconditioning2020}
and reference therein for more details as well as an analysis of the SCF
convergence on such systems.

For modelling the \ch{Fe2MnAl} alloy we employ the parameters of Section \ref{sec:setup}
as well as $\Ecut = 45$ Ha, Gaussian smearing with temperature $T = 10^{-2}$ Ha
and a $13 \times 13 \times 13$ $\bm{k}$-point grid.
This results in $\Ng = 250\,000$, $\Nocc = 24\,612$ and $\Nb \approx 4\,820$.
For the inexact GMRES we employ a restart size of $m=10$
as well as a Kerker preconditioner in order to capture the
Coulomb divergence in the metallic channel of this material.
\begin{table}
  \centering
  \smaller
  \begin{tabular}{c|cccc|cccccc}
    \toprule
    Strategy         & \textcolor{blue}{\sagr}    & \textcolor{blue}{\sbal}    & \sdten    & \sdhun    & \textcolor{blue}{\sPagr}    & \textcolor{blue}{\sPbal}   & \textcolor{blue}{\sPgrt}    & \sPdten    & \sPdhun   & \sPdnorm \\
    \midrule
    $\|\rend\|$     & 3e-8        & 4e-9        & 2e-6        & 1e-7        & 1e-8        & \textcolor{orange}{8e-10}  & \textcolor{orange}{3e-10}  & 2e-6        & 1e-7        & \textcolor{orange}{1e-9}   \\
    $\Nham$      & \textcolor{orange}{12\,M} & 15\,M      & 21\,M      & 22\,M      & \textcolor{orange}{10\,M} & \textcolor{orange}{12\,M} & 15\,M      & 16\,M      & 16\,M      & 20\,M      \\
    $\releta$ & \textcolor{orange}{2.06}             & 1.85                  & 1.00                  & 1.09                  & \textcolor{orange}{2.54}             & \textcolor{orange}{2.33}             & 1.96             & 1.33                  & 1.47                  & 1.42                  \\
    \bottomrule
  \end{tabular}
  \caption{\textbf{Heusler:} Accuracy of the returned solution $\|\rend\|$,
  total number of Hamiltonian applications $\Nham$ and relative efficiency
  $\releta$ (referencing \sdten) for the Heusler system with $\tau=10^{-9}$.
  Adaptive strategies are highlighted in blue, and the top three
  strategies for each metric are highlighted in orange.}
  \label{tab:summary_H_10_-9}
\end{table}
\begin{figure}
    \begin{center}
      \begin{overpic}[width=0.9\textwidth]{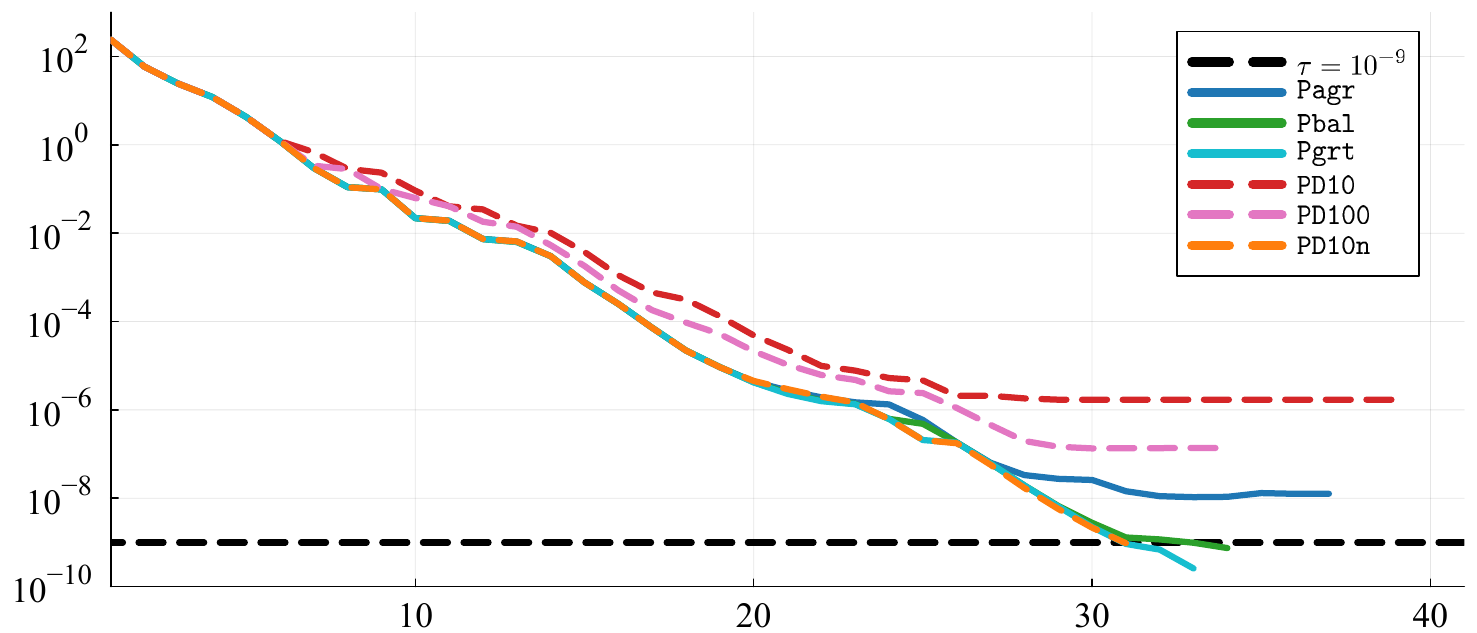}
          \put(47, 27) {\includegraphics[width=3.8cm]{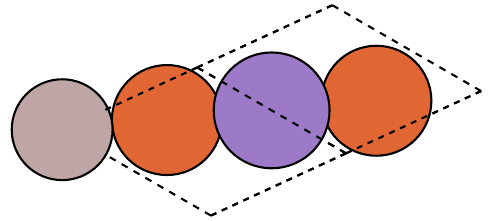}}
      \put(35,-2.5){GMRES iteration number $i$}
      \put(-1.5,10){\rotatebox{90}{true residuals $\| {\bm r}_i \|$}}
    \end{overpic}
    \end{center}
    \vspace*{2mm}
    \caption{\textbf{Heusler:} True residual norm $\|\bm r_i\|$ vs.~%
        GMRES iteration $i$ %
        for the \ch{Fe2MnAl} Heusler system.
        The inset shows the structure of \ch{Fe2MnAl}
        with iron (Fe) in orange, manganese (Mn) in violet
        and aluminium (Al) in grey.}
    \label{fig:convergence_H}
\end{figure}
We solve the Dyson equation until a tolerance of $\tau = 10^{-9}$
using a selection of strategies for selecting the CG tolerance,
see Table~\ref{tab:summary_H_10_-9} and Figure~\ref{fig:convergence_H}.
Comparing to the aluminium system the problem sizes
of the Dyson and Sternheimer equations ($\Ng$ and $\Nb$) are much smaller.
Still, the difficulty of this system is clearly visible
from the total number of Hamiltonian applications required to reach convergence
as well as the inability of the baseline strategies \sPdten and \sPdhun
to obtain accurate solutions --- with final true residuals being
two orders %
larger than the requested accuracy $\tau$.
For these two baseline strategies we also notice an early deviation
from ideal convergence in Figure~\ref{fig:convergence_H}.
In contrast our adaptive strategies \sPbal and \sPgrt stay close to
the baseline \sPdnorm accuracy
and keep a steady linear convergence.
We further remark that also for this Heusler system we observe
a superlinear convergence in the total number of Hamiltonian applications
for adaptive strategies similar to Figure~\ref{fig:convergence_Al40_new} (B).

For this system our \sPbal strategy again provides the best compromise.
It achieves both the desired accuracy of $\tau = 10^{-9}$
while at the same time requiring about 40\%
Hamiltonian applications less compared
to the equally accurate \sPdnorm baseline strategy
--- in absolute numbers a saving of about $8$ million Hamiltonian applications.
This demonstrates our adaptive CG tolerance strategies
to be reliably applicable also for challenging material systems.

\subsection{Semiconductors: Silicon}
We consider a simple silicon primitive unit cell
with two silicon atoms to demonstrate our adaptive CG tolerance strategies
to be applicable to semiconductors and insulators.
For the silicon systems we consider a $8 \times 8 \times 8$ $\bm k$-point grid
with $\Ecut = 40$ Ha and no smearing.
This results in $\Ng = 91\,125$, $\Nocc = 2\,048$ and $\Nb \approx 3\,020$.
In the context of SCF methods it is well known that for
a semiconducting material a Kerker preconditioner should not be employed
to obtain best convergence,
see e.g.~the discussion in~\cite{herbstBlackboxInhomogeneousPreconditioning2020}.
In line with Section~\ref{sec:kerker}
we thus also refrain from preconditioning the inexact GMRES for silicon.
For a GMRES convergence tolerance $\tau = 10^{-9}$
this results in the performance metrics of Table~\ref{tab:summary_silicon_20_-9},
whereas consistently worse results are obtained with Kerker preconditioning.
In agreement with the metallic systems our adaptive strategies turn out to be reliable,
all achieving a better accuracy than the prescribed tolerance $\tau = 10^{-9}$.
While \sagr and \sgrt turn out to be slightly less costly for this simple material,
\sbal still remains a good compromise saving about $30\%$ compared to \sdnorm
--- as the only baseline strategy also achieving the accuracy $\tau$.

\begin{table}
  \centering
  \smaller
  \begin{tabular}{c|cccccc}
    \toprule
    Strategy         & \textcolor{blue}{\sagr}    & \textcolor{blue}{\sbal}    & \textcolor{blue}{\sgrt}    & \sdten    & \sdhun    & \sdnorm \\
    \midrule
    $\|\rend\|$     & 8e-10       & \textcolor{orange}{3e-10}  & \textcolor{orange}{2e-10}  & 3e-8        & 3e-9        & \textcolor{orange}{2e-10}  \\
    $\Nham$      & \textcolor{orange}{263\,k} & \textcolor{orange}{301\,k} & \textcolor{orange}{300\,k} & 340\,k      & 379\,k      & 436\,k      \\
    $\releta$ & \textcolor{orange}{1.49}             & \textcolor{orange}{1.37}             & \textcolor{orange}{1.37}             & 1.00                  & 0.99                  & 0.95                  \\
    \bottomrule
  \end{tabular}
  \caption{\textbf{Silicon:} Accuracy of the returned solution $\|\rend\|$,
  total number of Hamiltonian applications $\Nham$ and relative efficiency
  $\releta$ (referencing \sdten) for silicon with $\tau=10^{-9}$.
  Adaptive strategies are highlighted in blue, and the top three
  strategies for each metric are highlighted in orange.}
  \label{tab:summary_silicon_20_-9}
\end{table}

\section{Conclusions}
\label{sec:conclusions}

The efficient simulation of linear response properties in density functional theory
requires the iterative solution of the Dyson equation.
This is a nested linear problem, in which each application of the operator
in turn requires solving a sizeable number of Sternheimer equations~(SE) iteratively.
In this work, we have performed an error analysis for the numerical solution of the
Dyson equation in plane-wave basis sets, quantifying rigorously the error propagation
from the approximate SE solutions to the final Dyson solution.
Based on this analysis and several physically motivated modifications, we have proposed
a set of strategies for adaptively selecting the convergence tolerance $\taucg_{i,n}$
when solving the SEs.
When employing an inexact GMRES framework to solve the Dyson equation
this allowed us to robustly tune the tightness of SE solution in each iteration
of the outer GMRES solver.
In non-trivial test cases on practically relevant materials systems
this enabled our recommended \sbal (balanced) strategy to obtain
a reduction of computational time by about 40\%
while maintaining the accuracy.
This is in contrast to baseline strategies, which feature a heuristically tuned
constant SE tolerance and which we could demonstrate to frequently fail
to achieve a sufficiently accurate solution.
Our framework combines well with standard preconditioning strategies in the field.
For example in combination with the
Kerker preconditioner (commonly employed in self-consistent field methods)
we could obtain an additional 20\% reduction of the cost
for solving the Dyson equation in metallic systems.

Overall this work presents a first approach for solving the Dyson
equation employing exclusively Krylov subspace methods.
Making use of the inexact Krylov subspace formalism our work offers
a reliable and efficient black-box solver
for DFPT without the need for any heuristic tuning of tolerances.
With the obtained efficiency gains our strategies thus provide
a viable alternative to the more wide-spread approaches based on
accelerated fixed-point methods~\cite{gonze1995adiabatic,baroni2001phonons,Gonze1997}
in the field.
However, in contrast to the latter Krylov-subspace methods are mathematically
well-studied, such that there are several promising directions to extend
the foundations of this work. 
First, when requiring only a low-accuracy SE solution
one could employ lower-precision arithmetics to provide further speedups
on modern GPU-based compute architectures.
Second, since the SEs form a sequence of closely related
linear systems, Krylov subspace recycling techniques could be explored.
Finally, our framework can be easily extended to the perturbation theory of other
mean-field models, such as Hartree-Fock or Gross-Pitaevskii equations, where
the Dyson equation is replaced by an alternative linearisation of the
underlying nonlinear eigenvalue problem.

\ifarXiv
    \appendix
    \section{Proof of inexact GMRES lemmas in Section~\ref{sec:inexact_gmres}}
\label{supp:proofs}
\begin{proof}[Proof of Lemma~\ref{lem:inexact_residual_norm}]
  By the inexact Arnoldi decomposition \eqref{eq:inexact_arnoldi_decomposition} and the definition of $\bm E_i = \widetilde{\bm{\mathcal E}}^{(i)} - \bm{\mathcal E}$, we have
  \begin{equation}
    \bm{\mathcal E} \bm V_m + \left[ \bm E_1 \bm v_1, \bm E_2 \bm v_2, \cdots, \bm E_m \bm v_m \right] = \bm V_{m+1} \bm H_m.
  \end{equation}
  Therefore, we can write the true residual $\bm r_m$ as
  \begin{equation}
      \begin{aligned}
          \bm r_m &= \bm b - \bm{\mathcal E} \bm x_m = \bm b - \bm{\mathcal E} \bm x_0 - \bm{\mathcal E} \bm V_m \bm y_m \\
          & = \bm b - \bm{\mathcal E} \bm x_0 - \bm V_{m+1} \bm H_m \bm y_m + \left[ \bm E_1 \bm v_1, \bm E_2 \bm v_2, \cdots, \bm E_m \bm v_m \right] \bm y_m \\
          & =  \bm E_0 \bm x_0 + \bm V_{m+1} \left( \beta \bm e_1 - \bm H_m \bm y_m \right) + \left[ \bm E_1 \bm v_1, \bm E_2 \bm v_2, \cdots, \bm E_m \bm v_m \right] \bm y_m, \\
      \end{aligned}
  \end{equation}
  where for the last equality, we used the fact that
  $\bm V_{m+1}\bm e_1 = \bm v_1 = \bm r_0/\beta$
  and $\bm r_0 = \bm b - \widetilde{\bm{\mathcal E}}^{(0)} \bm x_0$
  due to the Arnoldi process. The result then follows from the definition
  of estimated residual $\widetilde{\bm r}_m = \beta \bm e_1 - \bm H_m \bm y_m$,
  the orthonormality of $\bm V_{m+1}$, and the triangle inequality.
\end{proof}

\begin{proof}[Proof of Lemma~\ref{lem:gmres_y_k_bound}]
  We give a simpler proof of this result compared to the one provided in~\cite{simoncini2003theory}.
  Note that
  \begin{equation}
      \begin{aligned}
          \left[ \bm y_m \right]_i &= \bm e_i^* \bm y_m = \bm e_i^* \bm H_m^{\dagger} \beta \bm e_1 = \bm e_i^* \bm H_m^{\dagger} \beta \bm e_1 - \bm e_i^* \bm H_m^{\dagger} \begin{bmatrix} \bm H_{i-1} \bm y_{i-1} \\ \bm 0 \end{bmatrix} \\
          & = \bm e_i^* \bm H_m^{\dagger} \begin{bmatrix} \beta \bm e_1 - \bm H_{i-1} \bm y_{i-1} \\ \bm 0 \end{bmatrix} = \bm e_i^* \bm H_m^{\dagger} \begin{bmatrix} \widetilde{\bm r}_{i-1} \\ \bm 0 \end{bmatrix},
      \end{aligned}
  \end{equation}
  where $\bm H_m^{\dagger}$ is the pseudoinverse of $\bm H_m$. Therefore 
  \begin{equation}
      \big| \left[ \bm y_m \right]_i \big|
      \leq \| e_i^* \bm H_m^{\dagger} \| \, \| \widetilde{\bm r}_{i-1} \|
      \leq \dfrac{1}{\sigma_{m}(\bm H_m)} \| \widetilde{\bm r}_{i-1} \|.
  \end{equation}
\end{proof}

\section{Bound prefactor without eigenvalue gaps}
\label{supp:approx_gap}
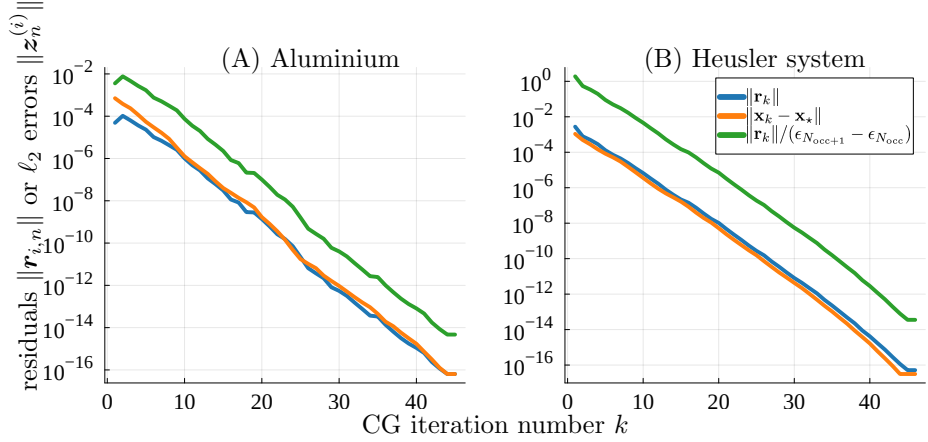
\begin{figure}
  \centering
   \begin{overpic}[width=0.9\textwidth]{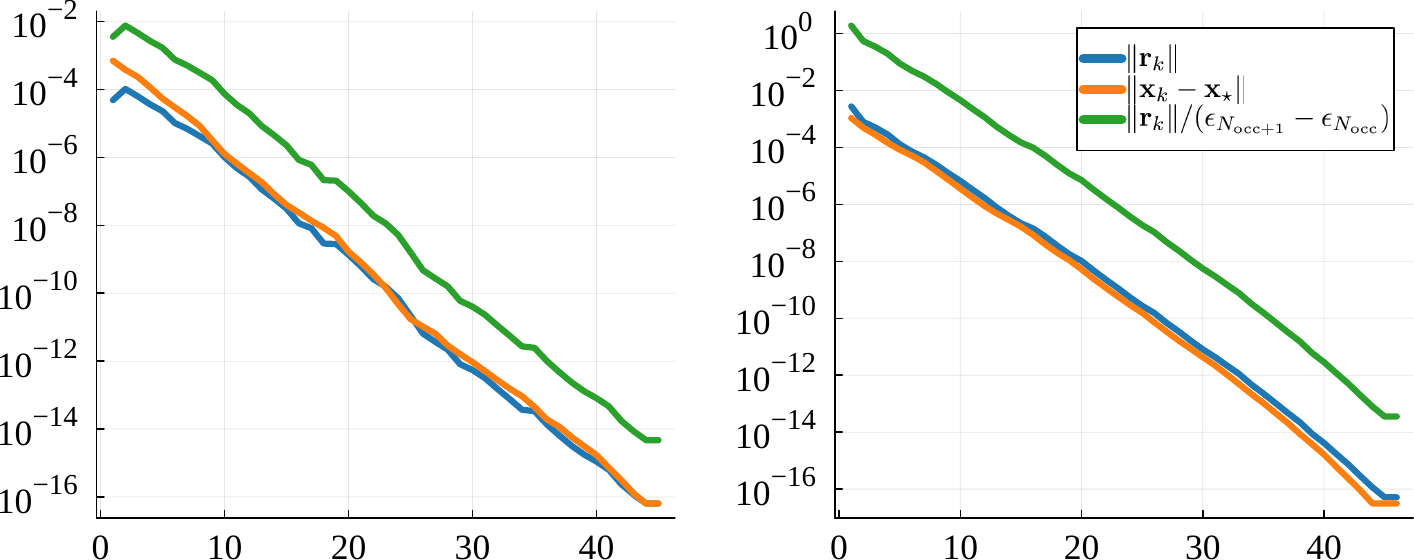}
      \put(20,38.8){(A) Aluminium}
      \put(68,38.8){(B) Heusler system}
      \put(36,-2.7){CG iteration number $k$}
      \put(-5,1){\rotatebox{90}{residuals $\| \bm{r}_{n,k}^{(i)} \|$ or $\ell_2$ errors $\| \bm z_{n,k}^{(i)} \|$}}
    \end{overpic}
    \vspace{2mm}
    \caption{Residual norms $\| \bm{r}_{n,k}^{(i)} \|$, $\ell_2$ errors $\| \bm z_{n,k}^{(i)} \|$
     and error bounds $\| \bm{r}_{n,k}^{(i)} \|/(\epsilon_{\Nocc + 1} - \epsilon_{\Nocc})$
    v.s. CG iteration number $k$ for the worst conditioned
        Sternheimer equation of (A) the aluminium system ($\ell = 10$)
      and (B) the Heusler system described in Section~\ref{sec:experiments}.}
    \label{fig:cg_z_residual}
\end{figure}
In Section \ref{sec:approx_gap} of the main text we argue
that the eigenvalue gap term $(\epsilon_{\Nocc + 1} - \epsilon_n)$
in the prefactor $C_{i,n}$ of Theorem~\ref{thm:inexact_gmres_dyson}
can be discarded, i.e.~set to $1$ in practical calculations
provided that appropriate techniques such as preconditioning
or the Schur complement approach~\cite{cances2023numerical} are employed.
Indeed, this can be demonstrated numerically, see Figure~\ref{fig:cg_z_residual}.
In this plot we show the $\ell_2$ norm of the error 
$\bm{z}_{n,k}^{(i)} = {\bm{\delta \phi^Q}_{n,k}}^{(i)} - \widetilde{\bm{\delta \phi^Q}_{n,k}}^{(i)}$
during CG iterations $k = 1, 2, \cdots, $ as well as the $\ell_2$ norm of the corresponding CG residual
$\bm{r}_{n,k}^{(i)} := \bm b_n - \bm A_n \widetilde{\bm{\delta \phi^Q}_{n,k}}^{(i)}$
and the residual divided by the gap. Here, the additional index $k$ denotes the CG iteration.
This is demonstrated for the Sternheimer equation
corresponding to the highest occupied eigenvalue $\epsilon_{\Nocc}$,
which is associated with the worst conditioning.
For (A) the aluminium supercell the bound closely follows the $\ell_2$ error
but sometimes undershoots while
for (B) the Heusler system the $\ell_2$ error is even bounded by the residual norm
--- see Section~\ref{sec:experiments} for more details on the setup of these
numerical experiments.

\section{Growth of Kernel-vector-product norm}
\label{supp:approx_kernel}
In Section \ref{sec:approx_kernel} of the main text we argue
that given appropriate preconditioning of the GMRES,
the Hartree-exchange-correlation kernel term $\| \mathbf{K} \mathbf{v}_i \|$
can be dropped from the prefactor $C_{i,n}$ of Theorem~\ref{thm:inexact_gmres_dyson}.
Indeed, Figure~\ref{fig:Kv_system_size} shows for the aluminium supercell,
that the appropriate Kerker preconditioner enables to overcome
the quadratic increase of this term with system size,
i.e.~that bound \eqref{eq:bound_k_naive} is far too pessimistic.
\begin{figure}
  \centering
  \begin{overpic}[width=0.9\textwidth]{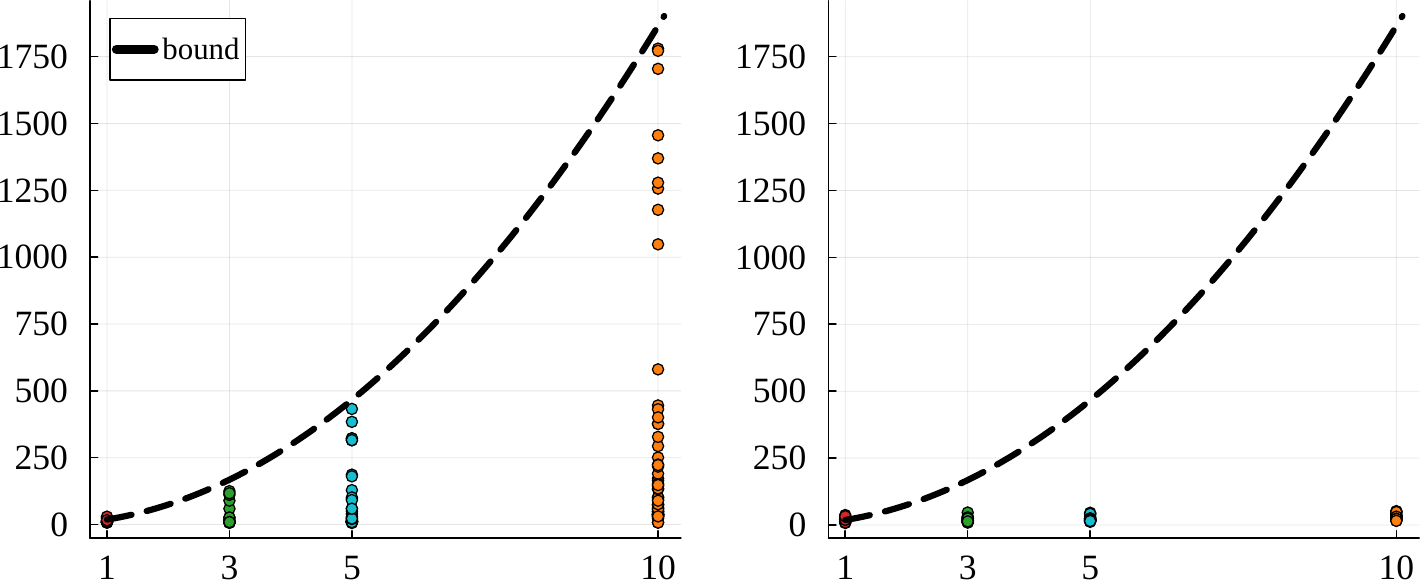}
    \put(20,38.8){(A) w/o Kerker}
    \put(70,38.8){(B) w/ Kerker}
    \put(16,-2.7){length of the Aluminium supercell $\ell$, i.e., relative size of $\|\bm a_1\|$}
    \put(-3.3,6){\rotatebox{90}{values or bound of $\| \bm K \bm v_i \|$}}
  \end{overpic}
  \vspace*{2mm}
  \caption{
  Values of $\| \bm K \bm v_i \|$ (circles) during the GMRES iterations for Aluminium supercell systems of different sizes
  (A) without and (B) with Kerker preconditioning.
  We employ the \sagr strategy to set the CG tolerances (see Table~\ref{tab:strategies})
  and use $\ell$ to denote the factor by which $\bm a_1$ grows relative to the smallest Aluminium system,
  see Section~\ref{sec:Al} for details on the computational setup.}
  \label{fig:Kv_system_size}
\end{figure}

\section{Variation of CG iterations and tolerance across GMRES steps}
\label{supp:iterations}
\begin{figure}
    \begin{center}
    \begin{overpic}[width=0.9\textwidth]{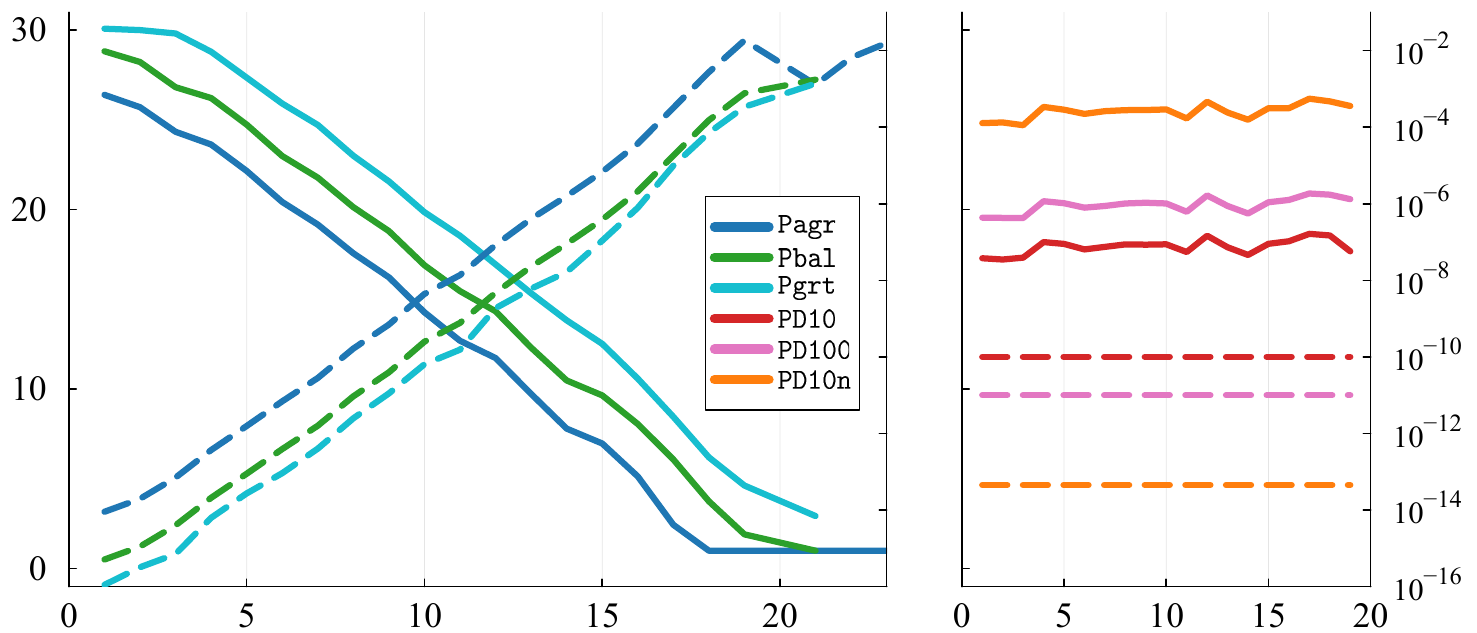}
      \put(36,-2.5){GMRES iteration number $i$}
      \put(-2.5,4){\rotatebox{90}{average \# CG iterations (solid)}}
      \put(100,41){\rotatebox{-90}{average CG tolerances (dashed)}}
    \end{overpic}
    \end{center}
    \vspace*{2mm}
    \caption{Average number of CG iterations (in solid lines)
        and geometric average of CG tolerances (in dashed lines) per GMRES iteration
        of the aluminium supercell system discussed in Section~\ref{sec:Al}.
    }
  \label{fig:cg_iter_tol_Al40}
\end{figure}
Figure~\ref{fig:cg_iter_tol_Al40} shows how the CG tolerances employed
for solving the Sternheimer equations evolve during the GMRES iterations.
With respect to averaging over all $(n, \bm k)$ we show both the
average tolerance $\taucg_{i,n\bm k}$ as well as the average number of CG iterations
per GMRES iteration $i$ for the preconditioned strategies.
Note that one Hamiltonian application is needed per CG iteration.
For the baseline strategies (right hand side) the CG tolerances are fixed constants
and thus the CG iterations remain almost constant across the GMRES iterations.
In contrast, our adaptive strategies (left hand side)
require looser CG tolerances as the GMRES progresses and notably,
initially employ tighter CG tolerances than the baselines,
but near convergence only require $1$ or $2$ CG iterations for each SE.
This behaviour suggests that the accuracy of the first few matrix-vector
products effectively limits the final solution accuracy,
which cannot be compensated by computing the matrix-vector product $\bm{\mathcal E} \bm v$
more precisely in later GMRES iterations.
\section{Testing size and GMRES parameter dependence}
\label{supp:aluminium_parameters}
With respect to the aluminium system already considered in Section~\ref{sec:Al}
we study the performance of the CG tolerance strategies from Table~\ref{tab:strategies}
as the system size, the Krylov subspace $m$ and the targeted GMRES tolerance $\tau$ are varied.
We keep the computational setup described in Section~\ref{sec:Al}
and consider aluminium supercells in which the $4$-atom cubic unit cell
is repeated $\ell$ times along the $x$ direction where $\ell = 1, 3, 5, 10$.
For $\ell = 10$ we thus obtain exactly the system studied before.
We use the same parameters as in Section~\ref{sec:Al} with the exception of the
unit cell $\ell=1$ where we set the $\bm k$-point grid to $3 \times 3 \times 3$.
Additionally, on the extremely simple unit cell system ($\ell=1$)
we do not test \sPgrt and \sPdnorm strategies.

We test GMRES restart sizes of $m=10$ and $20$, and do not consider smaller values.
Note that even storing $m=20$ Krylov vectors --- amounting to $m\Ng \approx 360 \Nb$
FP64 numbers --- is a small memory cost compared to the amount of storage required
for the occupied orbitals.
The latter are complex-valued and thus require $2 \Nocc \Nb \approx 1100 \Nb$
FP64 numbers for the largest system with $\ell = 10$.
For the GMRES tolerance, we focus on $\tau = 10^{-6}$ and $10^{-9}$, 
as other values are of limited relevance for practical applications.
Our results are summarised in Tables~\ref{tab:summary_Al_re},
\ref{tab:summary_Al_res} and \ref{tab:summary_Al_cg}
with an even more comprehensive study available in the file
\texttt{Al/data\_logs/summary\_tables.md} in our GitHub repository\footnote{\url{https://github.com/bonans/inexact_Krylov_response}}.
Similar studies for the Heusler system are also available in the file
\texttt{Fe2MnAl/data\_logs/summary\_tables.md} in the same repository.
Throughout this study our adaptive strategies consistently outperform the
baseline approaches with \sPbal providing the best compromise
between reliably achieving the desired tolerance $\tau$
and required number of Hamiltonian applications.

\begin{table}[H]
  \centering
  \begin{tabular}{c|ccc|ccc}
      \toprule
      $\ell$, $m$, $\tau$                  & \textcolor{blue}{\sPagr} & \textcolor{blue}{\sPbal} & \textcolor{blue}{\sPgrt} & \sPdten & \sPdhun & \sPdnorm \\
      \midrule
      \textcolor{white}{0}1, 10, $10^{-9}$ & \textcolor{orange}{1.63}        & \textcolor{orange}{1.55}         &  & 1.00          & 1.00           &              \\
      \textcolor{white}{0}3, 10, $10^{-9}$ & \textcolor{orange}{1.60}        & \textcolor{orange}{1.52}         & \textcolor{orange}{1.35}        & 1.00          & 1.00           & 0.99             \\
      \textcolor{white}{0}5, 10, $10^{-9}$ & \textcolor{orange}{1.47}        & \textcolor{orange}{1.51}         & \textcolor{orange}{1.35}        & 1.00          & 0.99           & 0.98             \\
      \midrule
      10, 10, $10^{-9}$                    & \textcolor{orange}{1.61}        & \textcolor{orange}{1.55}         & \textcolor{orange}{1.40}        & 1.00          & 1.00           & 0.98             \\
      10, 20, $10^{-9}$                    & \textcolor{orange}{1.53}        & \textcolor{orange}{1.47}         & \textcolor{orange}{1.32}        & 1.00          & 0.99           & 0.98             \\
      \midrule
      10, 10, $10^{-6}$                    & \textcolor{orange}{1.59}        & \textcolor{orange}{1.54}         & \textcolor{orange}{1.31}        & 1.00          & 1.04           & 0.98             \\
      10, 20, $10^{-6}$                    & \textcolor{orange}{1.43}        & \textcolor{orange}{1.39}         & \textcolor{orange}{1.26}        & 1.00          & 1.04           & 0.99             \\
      \bottomrule
  \end{tabular}
  \caption{Relative efficiency $\releta$ (referencing \sPdten) 
  for the aluminium systems with varied $\ell$, $m$ and $\tau$.
  Our adaptive strategies are highlighted in blue, 
  and the top three strategies for each metric are
  highlighted in orange.}
  \label{tab:summary_Al_re}
\end{table}
\begin{table}[H]
    \centering
    \footnotesize
    \begin{tabular}{c|ccc|ccc}
        \toprule
        $\ell$, $m$, $\tau$                  & \textcolor{blue}{\sPagr} & \textcolor{blue}{\sPbal}             & \textcolor{blue}{\sPgrt}                                  & \sPdten            & \sPdhun           & \sPdnorm                             \\
        \midrule
        \textcolor{white}{0}1, 10, $10^{-9}$ & $1.37 \!\cdot\! 10^{-9}$        & \textcolor{orange}{$2.96 \!\cdot\! 10^{-10}$} &  & $1.43 \!\cdot\! 10^{-8}$ & $1.19 \!\cdot\! 10^{-9}$ &  \\
        \textcolor{white}{0}3, 10, $10^{-9}$ & $1.08 \!\cdot\! 10^{-8}$        & \textcolor{orange}{$7.73 \!\cdot\! 10^{-10}$}                    & \textcolor{orange}{$3.27 \!\cdot\! 10^{-10}$}                    & $2.46 \!\cdot\! 10^{-7}$ & $2.06 \!\cdot\! 10^{-8}$ & \textcolor{orange}{$6.92 \!\cdot\! 10^{-10}$}                    \\
        \textcolor{white}{0}5, 10, $10^{-9}$ & $1.12 \!\cdot\! 10^{-8}$        & \textcolor{orange}{$5.99 \!\cdot\! 10^{-10}$}                    & \textcolor{orange}{$1.36 \!\cdot\! 10^{-10}$}                    & $8.01 \!\cdot\! 10^{-7}$ & $7.63 \!\cdot\! 10^{-8}$ & \textcolor{orange}{$1.32 \!\cdot\! 10^{-9\textcolor{white}{0}}$} \\
        \midrule
        10, 10, $10^{-9}$                    & $4.34 \!\cdot\! 10^{-8}$        & \textcolor{orange}{$1.50 \!\cdot\! 10^{-9\textcolor{white}{0}}$} & \textcolor{orange}{$1.94 \!\cdot\! 10^{-10}$}                    & $8.69 \!\cdot\! 10^{-6}$ & $6.96 \!\cdot\! 10^{-7}$ & \textcolor{orange}{$3.16 \!\cdot\! 10^{-9\textcolor{white}{0}}$} \\
        10, 20, $10^{-9}$                    & $1.74 \!\cdot\! 10^{-8}$        & \textcolor{orange}{$7.17 \!\cdot\! 10^{-10}$}                    & \textcolor{orange}{$2.41 \!\cdot\! 10^{-10}$}                    & $1.33 \!\cdot\! 10^{-5}$ & $1.33 \!\cdot\! 10^{-6}$ & \textcolor{orange}{$5.32 \!\cdot\! 10^{-9\textcolor{white}{0}}$} \\
        \midrule
        10, 10, $10^{-6}$                    & $7.03 \!\cdot\! 10^{-5}$        & \textcolor{orange}{$1.12 \!\cdot\! 10^{-6\textcolor{white}{0}}$} & \textcolor{orange}{$2.41 \!\cdot\! 10^{-7\textcolor{white}{0}}$} & $8.54 \!\cdot\! 10^{-3}$ & $7.44 \!\cdot\! 10^{-4}$ & \textcolor{orange}{$4.12 \!\cdot\! 10^{-6\textcolor{white}{0}}$} \\
        10, 20, $10^{-6}$                    & $2.98 \!\cdot\! 10^{-5}$        & \textcolor{orange}{$7.71 \!\cdot\! 10^{-7\textcolor{white}{0}}$} & \textcolor{orange}{$6.38 \!\cdot\! 10^{-8\textcolor{white}{0}}$} & $1.32 \!\cdot\! 10^{-2}$ & $1.09 \!\cdot\! 10^{-3}$ & \textcolor{orange}{$6.29 \!\cdot\! 10^{-6\textcolor{white}{0}}$} \\
        \bottomrule
    \end{tabular}
    \caption{Accuracy of the returned solution
    $\|\rend\|$ for the aluminium systems with varied $\ell$, $m$ and $\tau$.
    Our adaptive strategies are highlighted in blue, 
  and the top three strategies for each metric are
  highlighted in orange.}
    \label{tab:summary_Al_res}
\end{table}
\begin{table}[H]
  \centering
  \begin{tabular}{c|ccc|ccc}
      \toprule
      $\ell$, $m$, $\tau$                  & \textcolor{blue}{\sPagr}                 & \textcolor{blue}{\sPbal}                & \textcolor{blue}{\sPgrt}                 & \sPdten                                   & \sPdhun               & \sPdnorm            \\
      \midrule
      \textcolor{white}{0}1, 10, $10^{-9}$ & \textcolor{white}{0}\textcolor{orange}{17\,k} & \textcolor{white}{0}\textcolor{orange}{19\,k} &  & \textcolor{white}{0}25\,k                     & \textcolor{white}{0}28\,k  &  \\
      \textcolor{white}{0}3, 10, $10^{-9}$ & \textcolor{white}{0}\textcolor{orange}{35\,k} & \textcolor{white}{0}\textcolor{orange}{41\,k} & \textcolor{white}{0}\textcolor{orange}{47\,k} & \textcolor{white}{0}49\,k                     & \textcolor{white}{0}55\,k  & \textcolor{white}{0}63\,k \\
      \textcolor{white}{0}5, 10, $10^{-9}$ & \textcolor{white}{0}\textcolor{orange}{80\,k} & \textcolor{white}{0}\textcolor{orange}{86\,k} & 101\,k                                        & \textcolor{white}{0}\textcolor{orange}{97\,k}                     & 109\,k                     & 130\,k                    \\
      \midrule
      10, 10, $10^{-9}$                    & \textcolor{orange}{183\,k}                    & \textcolor{orange}{215\,k}                    & 253\,k                                        & \textcolor{orange}{233\,k}                    & 263\,k                     & 332\,k                    \\
      10, 20, $10^{-9}$                    & \textcolor{orange}{168\,k}                    & \textcolor{orange}{196\,k}                    & 228\,k                                        & \textcolor{orange}{194\,k}                    & 219\,k                     & 275\,k                    \\
        \midrule
      10, 10, $10^{-6}$                    & \textcolor{white}{0}\textcolor{orange}{96\,k} & \textcolor{orange}{121\,k}                    & 152\,k                                        & \textcolor{orange}{113\,k}                    & 127\,k                     & 180\,k                    \\
      10, 20, $10^{-6}$                    & \textcolor{orange}{101\,k}                    & 124\,k                                        & 152\,k                                        & \textcolor{white}{0}\textcolor{orange}{99\,k}                     & \textcolor{orange}{113\,k} & 159\,k                    \\
      \bottomrule
  \end{tabular}
  \caption{Total number of Hamiltonian applications $\Nham$ 
  Our adaptive strategies are highlighted in blue, 
  and the top three strategies for each metric are
  highlighted in orange.}
  \label{tab:summary_Al_cg}
\end{table}

\section{Adaptive tolerance selection when modelling perfect crystals}
\label{supp:perfect_crystals}
In section~\ref{sec:crystals} we discussed the extension of
density-functional perturbation theory to the case of perfect crystals.
The application of $\chi_0$ to a perturbation $\delta V$ is shown
in equation \eqref{eq:bzchi0}, which in the discretised setting is written as
\begin{equation}
    \label{eq:kchi0}
  \bm \chi_0 \bm{\delta V}
  = \sum_{\bm k =1}^{\Nmp} w_{\bm k} \sum_{n=1}^{\Nocck} 2 f_{n\bm k} \Real \Big( (\bm F_{\bm k}^{-1}\bm{\phi}_{n\bm k})^* \odot ( \bm F_{\bm k}^{-1} \bm{\delta \phi}_{n\bm k}) \Big) + \delta f_{n\bm k} \, \left| \bm F_{\bm k}^{-1} \bm{\phi}_{n \bm k} \right|^2.
\end{equation}
Here, we explicitly exploit crystal symmetry,
which enables to reduce the original sum over all $L = |\mpgrid|$ $\bm k$-points
to a sum over the $\Nmp \leq L$ unique irreducible $\bm k$-points
and with a slight abuse of notation we further employ the integers
$1$ to $\Nmp$ to label the irreducible $\bm k$-points.
As a result of the exploitation of symmetry we further need to introduce
an appropriate weight factor $w_{\bm k}$,
which equals to the size of the orbit of $\bm k$
(number of symmetry-equivalent $\bm k$-points generated by this irreducible $\bm k$-point)
divided by $L$. We also recall that the Fourier
transforms differ across $\bm k$-points, see the first paragraph 
of section \ref{sec:Al} for a more detailed description.

To obtain the modified expressions for the prefactor
$C_{i,n {\bm k}}$ we modify the proof of Lemma \ref{lem:cgtol_E}
when employing the expression \eqref{eq:kchi0} for $\bm \chi_0 \bm{\delta V}$.
Let $\bm{\Phi}_{\bm k} = [\bm{\phi}_{1 \bm k}, \cdots,\allowbreak \bm{\phi}_{\Nocck \bm k}]$ and $\bm{Y}_{\bm k} = [f_{1 \bm k} \bm z_{1 \bm k}^{(i)}, \cdots, f_{\Nocck \bm k} \bm z_{\Nocck \bm k}^{(i)}]$.
Then, the equivalent expression to
equation \eqref{eq:dielectricdiff} becomes
\begin{equation}
      \begin{aligned}
          \| (\bm{\mathcal E} - \widetilde{\bm{\mathcal E}}^{(i)}) \bm{v}_i\| &= 2 \left\| \bm K \bm{v}_i\right\| \left\| \sum_{\bm k =1}^{\Nmp} w_{\bm k} \sum_{n=1}^{\Nocck} f_{n\bm k} \Real \left( \left(\bm F_{\bm k}^{-1} \bm{\phi}_{n \bm k}\right)^* \odot \left( \bm F_{\bm k}^{-1} \bm{z}_{n\bm k}^{(i)} \right) \right) \right\|                                                             \\
          & = 2 \left\| \bm K \bm{v}_i\right\| \left\| \sum_{\bm k =1}^{\Nmp} w_{\bm k} \Real\left(\bm{F}_{\bm k}^{-1}\bm{\Phi}_{\bm k} \odot \bm{F}_{\bm k}^{-1}\bm{Y}_{\bm k}\right) \cdot \bm{1}_{\Nocck} \right\| \\
          & \leq 2 \left\| \, \bm K \bm{v}_i\right\| \sum_{\bm k =1}^{\Nmp} w_{\bm k}\|\bm{F}_{\bm k}^{-1}\bm{\Phi}_{\bm k}\|_{2,\infty} \, \|\bm{F}_{\bm k}^{-1}\bm{Y}_{\bm k}\|_{1,2} \sqrt{\Nocck}\\
          & \leq 2 w^{-1} \left\| \, \bm K \bm{v}_i\right\| \sum_{\bm k =1}^{\Nmp} w_{\bm k} \sqrt{\Nocck}\|\bm{F}_{\bm k}^{-1}\bm{\Phi}_{\bm k}\|_{2,\infty} \max_{1\leq n \leq \Nocck} f_{n \bm k} \|\bm{z}_{n\bm k}^{(i)}\|.
      \end{aligned}
\end{equation}
In Theorem \ref{thm:inexact_gmres_dyson} we thus replace the prefactor $C_{i,n}$ by
\begin{equation}
    C_{i,n\bm k} = \dfrac{\sqrt{|\Omega|}(\epsilon_{\Nocck + 1,\bm k} - \epsilon_{n\bm k})}{2 f_{n {\bm k}} w_{\bm k} \left\| \bm K \bm v_i \right\| \left\|\bm{F}_{\bm k}^{-1}\bm{\Phi}_{\bm k} \right\|_{2,\infty} \Nmp \sqrt{\Ng\Nocck} },
\end{equation}
where the $1 / \Nmp$ factor appears from an equal distribution of error
across the irreducible $\bm k$-points.
While the \sagr remains unchanged (all prefactors dropped)
the strategies \sgrt and \sbal should be changed accordingly
as shown in Table \ref{tab:strategies_k}.

\begin{table}
  \newcommand{\rhsterm}{$\dfrac{\sigma_m(\bm H_m)}{3m} \dfrac{ \tau}{\|\widetilde{\bm r}_{i-1}\|}$}
  \centering
  \smaller
  \begin{tabular}{crl}
    \toprule
    \multicolumn{3}{l}{Adaptive strategies (perfect crystals)\quad$\taucg_{i,n\bm k} = $} \\
    \midrule
    \sgrt     & $ \dfrac{1}{\left\| \bm K \bm v_i \right\| \left\| \Real(\bm{F}_{\bm k}^{-1}\bm{\Phi}_{\bm k}) \right\|_{2,\infty}}$
                $ \dfrac{\sqrt{|\Omega|}}{2 f_{n\bm k} w_{\bm k} \Nmp \sqrt{\Ng \, \Nocck}}$
                & \rhsterm \\[0.2em]
    \sbal     & $ \dfrac{\sqrt{|\Omega|}}{\sqrt{\Nocck}}$
    $ \dfrac{\sqrt{|\Omega|}}{2 f_{n\bm k} w_{\bm k} \Nmp \sqrt{\Ng \, \Nocck}}$
    & \rhsterm \\[0.2em]
    \bottomrule
  \end{tabular}
  \caption{Summary of strategies in case of perfect crystals.}
  \label{tab:strategies_k}
\end{table}

In our numerical tests we further make the reasonable assumption
that each $\Nocck$ is roughly the same and replace
$\Nocck \Nmp$ by $\Nocc := \sum_{\bm k}^{\Nmp} \Nocck$ for \sbal,
respectively 
$\sqrt{\Nocck} \Nmp$ by $\sqrt{\Nocc}$ for \sgrt
(dropping an additional $\sqrt{\Nmp}$).

\fi

\section*{Acknowledgments}
We would like to acknowledge fruitful discussions with Peter Benner,
Gaspard Kemlin and Xiaobo Liu.

\bibliographystyle{siamplain}
\bibliography{references}

\end{document}